\documentclass[reqno,letterpaper,11pt]{amsart}
\usepackage[margin=1.25in,footskip=0.25in]{geometry}

\usepackage{hyperref}

\usepackage{amsmath}
\RequirePackage{amsthm}

\usepackage{amsfonts,amssymb,amsmath,latexsym,wasysym,mathrsfs,bbm,stmaryrd}
\usepackage[all]{xy}
\usepackage[usenames]{color}
\usepackage{epsfig}

\usepackage{graphicx}
\usepackage{subcaption}

\usepackage{amsthm}
\usepackage{thmtools}
\usepackage{thm-restate}
\declaretheorem[numberwithin=section]{theorem}
\declaretheorem[sibling=theorem]{proposition}
\declaretheorem[sibling=theorem]{definition}

\declaretheorem[sibling=theorem]{lemma}

\declaretheorem[sibling=theorem]{remark}

\numberwithin{equation}{section}

\usepackage[mathcal]{euscript}

\newcommand{\deq}{\mathrel{\mathop:}=}

\newcommand{\C}{\mathbb{C}}

\newcommand{\R}{\mathbb{R}}
\newcommand{\E}{\mathbb{E}}

\newcommand{\bbP}{\mathbb{P}}

\newcommand{\tr}{\operatorname{tr}}
\newcommand{\Tr}{\operatorname{Tr}}

\long\def\symbolfootnote[#1]#2{\begingroup%
	\def\thefootnote{\fnsymbol{footnote}}\footnote[#1]{#2}\endgroup}

\begin{document}

\title{Deformed single ring theorems}
\author{
Ching-Wei Ho and Ping Zhong
}

\date{\today}
\address{Ching-Wei Ho: Institute of Mathematics, Academia Sinica, Taipei 10617, Taiwan}
\email{chwho@gate.sinica.edu.tw}
\address{Ping Zhong: Department of Mathematics, University of Houston, Houston, 
TX 77004, U.S.A.}
\email{pzhong@central.uh.edu}

\maketitle

\begin{abstract}
    Given a sequence of deterministic matrices $A = A_N$ and a sequence of deterministic nonnegative matrices $\Sigma=\Sigma_N$ such that $A\to a$ and $\Sigma\to \sigma$ in $\ast$-distribution for some operators $a$ and $\sigma$ in a finite von Neumann algebra $\mathcal{A}$. Let $U =U_N$ and $V=V_N$ be independent Haar-distributed unitary matrices. We use free probability techniques to prove that, under mild assumptions, the empirical eigenvalue distribution of $U\Sigma V^*+A$ converges to the Brown measure of $T+a$, where $T\in\mathcal{A}$ is an $R$-diagonal operator freely independent from $a$ and $\vert T\vert$ has the same distribution as $\sigma$. The assumptions can be removed if $A$ is Hermitian or unitary. By putting $A= 0$, our result removes a regularity assumption in the single ring theorem by Guionnet, Krishnapur and Zeitouni. We also prove a local convergence on optimal scale, extending the local single ring theorem of Bao, Erd\H{o}s and Schnelli.
\end{abstract}

\tableofcontents

\section{Introduction}

\subsection{Sum of two random matrices}
The study of the eigenvalues of the sum of two matrices is an old problem. 
The Horn's problem \cite{Horn1962} studies the possible eigenvalues of the sum of two deterministic Hermitian matrices with prescribed eigenvalues. Horn's conjecture \cite{Horn1962} was solved by Knutson and Tao \cite{KnutsonTao1999}, and  Knutson, Tao and Woodward \cite{KnutsonTaoWoodward2004}. In random matrix theory, it is a fundamental question to study the limiting eigenvalue distribution of the sum of two random matrices, one of which satisfies some symmetry assumptions. In the case where the two random matrices are Hermitian, if the empircal eigenvalue distributions of the two random matrices converge to some probability measures $\mu$ and $\nu$ on $\R$, the limiting eigenvalue distribution of this sum is much more specific than the solutions of the Horn's problem.
In terms of the macroscopic distribution, the empirical eigenvalue distribution of the sum of the two Hermitian random matrices typically converges to the free convolution $\mu\boxplus\nu$ of $\mu$ and $\nu$ (see \cite{Voiculescu1991} and \cite[Chapter 4]{MingoSpeicherBook}).
In fact, the seminal work of Voiculescu \cite{Voiculescu1991} suggests that free probability theory is a suitable framework to study the
asymptotic spectrum of the sum of large independent Hermitian or non-Hermitian random matrices.

A well-studied non-Hermitian random matrix model is the single ring random matrix model defined as follows. Consider two sequences of independent Haar-distributed unitary matrices $U=U_N$ and $V=V_N$ and a sequence of deterministic nonnegative matrices $\Sigma = \Sigma_N$ such that the eigenvalue distribution of $\Sigma$ converges weakly to a certain probability measure on $\R$. Free probability theory again provides natural limit operators for $U_N\Sigma_N V_N^*$.  By Voiculescu's asymptotic freeness result \cite{Voiculescu1991} (see also \cite{CollinsMale2014}), the random matrices $\{ U_N \Sigma_N V_N^* \}_{N=1}^\infty$ converge in $*$-moments to an $R$-diagonal operator introduced by Nica--Speicher \cite{NicaSpeicher-Rdiag}. The Brown measure 
\cite{Brown1986} is an analogue of the eigenvalue distribution of matrices for operators. 
 The Brown measure of $R$-diagonal operators was calculated by Haagerup--Larsen \cite{HaagerupLarsen2000} and Haagerup--Schultz \cite{HaagerupSchultz2009} (see also \cite{BSS2018, Zhong2022Rdiag} for alternative proofs). The Brown measure of any $R$-diagonal operator is supported in an annulus. In the physics literature, Feinberg--Zee \cite{FeinbergZee1997a} proved the single ring theorem which states that the limiting eigenvalue distribution of the matrix $U \Sigma V^*$ converges to a certain rotation-invariant probability measure whose support is an annulus. Guionnet, Krishnapur and Zeitouni \cite{GuionnetKZ-single-ring} then proved the single ring theorem rigorously under some technical assumptions. A major assumption in \cite{GuionnetKZ-single-ring} is to estimate the least singular value of the non-Hermitian random matrix model, a notoriously difficult problem in non-Hermitan random matrix theory. In their seminal work \cite{RudelsonVershynin2014}, Rudelson and Vershynin solved this problem for the single ring random matrix model along with other results, and hence removed this major assumption in \cite{GuionnetKZ-single-ring} concerning the least singular value.

In this article, in addition to the matrices $U$, $V$, $\Sigma$ introduced previously, we consider another deterministic matrix $A=A_N$ such that given any polynomial $P$ in two noncommuting indeterminates, the limit
\[\lim_{N\to\infty}\tr[P(A, A^*)]\]
exists, where $\tr = (1/N)\Tr$ is the normalized trace on $N\times N$ matrices. We then prove that, under mild assumptions, the empirical eigenvalue distribution of the random matrix
\[Y=U\Sigma V^*+A\]
converges to the Brown measure of a sum of two free random variables, one of which is $R$-diagonal. We call this result a deformed single ring theorem.  

Let $(\mathcal{A},\tau)$ be a $W^*$-probability space where $\mathcal{A}$ is a finite von Neumann algebra equipped with a faithful, normal, tracial state $\tau$. In the random matrix model described in the preceding paragraph, the deterministic matrix $A$ and the random matrix $U\Sigma V^*$ are asymptotically free \cite[Theorem 9 on page 105]{MingoSpeicherBook}: there exist $a, T\in \mathcal{A}$ that are freely independent in the sense of Voiculescu \cite{Voiculescu1985} such that for any polynomial $P$ in four noncommuting indeterminates
\[\lim_{N\to\infty}\E\tr[P(A, A^*,U\Sigma V^*, (U\Sigma V^*)^*)] =\tau[P(a,a^*,T,T^*)].\]
In particular, $Y = U\Sigma V^*+A$ converges in $\ast$-distribution to $y = T+a$, where $T$ is an $R$-diagonal operator. The Brown measure of $y$ is a natural candidate of the limiting eigenvalue distribution of $Y$. However, since the Brown measure is not continuous with respect to $\ast$-moments, knowing that $Y$ converges to $y$ in $\ast$-distribution does not mean the empirical eigenvalue distribution of $Y$ converges to the Brown measure of $y$. The purpose of this paper is to prove that the limiting eigenvalue distribution of $Y$ is indeed the Brown measure of $y$.

There is another non-Hermitian random matrix model closely related to the model that we consider in this paper. Ginibre \cite{Ginibre1965} analyzed the limiting eigenvalue distribution of the normalized random matrix with independent Gaussian entries with zero mean and unit variance, now called the Ginibre matrix. After decades of development (for example, \cite{Bai1997,Bordenave-Chafai-circular}), Tao--Vu \cite{TaoVu2010-aop} showed that the limiting eigenvalue distribution of the Ginibre matrix remains unchanged if the Gaussian entries are replaced by i.i.d. random variables with arbitrary distributions with zero mean and unit variance. \'Sniady \cite{Sniady2002} studied the limiting eigenvalue distribution of the sum of a Ginibre matrix and a determinsitic matrix. \'Sniady identifies that the limiting eigenvalue distribution of the sum of an arbitrary matrix and the Ginibre matrix is the Brown measure of the sum of two freely independent operators; one of these two operators is a Voiculescu's circular operator, which is $R$-diagonal. While the random matrix model $Y$ being considered in this paper is different from the one \'Sniady considered, if the law of the $\Sigma$ in this paper converges to the quarter-circular law, the limiting eigenvalue distribution of $Y$ is the same as that of the model in \'Sniady's paper. The Tao--Vu's replacement principle \cite[Corollary 1.8]{TaoVu2010-aop} shows that the empirical eigenvalue distribution of the sum of a random matrix with i.i.d. entries and a deterministic matrix converges to the same Brown measure. This Brown measure has an explicit formula computed in \cite{BordenaveCC2014cpam, HoZhong2020Brown, Ho2022elliptic, Zhong2021Brown, BelinschiYinZhong2022brown}.

In \cite{ErdosSchleinYau2009a,ErdosSchleinYau2009b}, Erd\H{o}s, Schlein and Yau investigated the local behavior of the Wigner random matrices, called a local law; they looked at the number of eigenvalues in an interval of length of order $\log N/N$ inside the bulk. Kargin \cite{Kargin2015} and Bao--Erd\H{o}s--Schnelli \cite{BaoErdosSchnelli2017} proved a local law of the sum of two large Hermitian random matrices, one of which is conjugated by a Haar-distributed unitary matrix. In the non-Hermitian framework, Bourgade, Yau and Yin \cite{BourgadeYauYin2014,BourgadeYauYin2014b} and Yin \cite{Yin2014} proved a local law for the normalized non-Hermitian random matrix model with i.i.d. entries. A local law for the random matrix $U\Sigma V^*$ was establshed by Benaych-Georges \cite{BenaychGeorges2017}; this result was later improved by Bao--Erd\H{o}s--Schnelli \cite{BaoES2019singlering} to an optimal scale. They call this result a local single ring theorem. In this paper, we also prove a deformed local single ring thoerem, studying the local behavior of the eigenvalues of $Y = U\Sigma V^*+A$ in the bulk.

\subsection{The random matrix model and the main results}
We first introduce the random matrix model considered in this paper. Suppose that
\begin{enumerate}
	\item $U=U_N$ and $V = V_N$ are independent Haar-distributed unitary matrices;
	\item $\Sigma=\Sigma_N$ are deterministic nonnegative diagonal matrices such that the eigenvalue distribution of $\Sigma$ converges weakly to a probability measure on $\R$ that is not a Dirac delta measure at $0$;
	\item $A = A_N$ are deterministic matrices such that all the $\ast$-moments of $A$ converge;
	\item there is a constant $M>0$ independent of $N$ such that
		 \begin{equation}
			\label{eq:Mmeaning}
			\|\Sigma\|,\|A\|\leq M;
		 \end{equation}
\end{enumerate}
We emphasize that all the assumptions that we make for the random matrix $Y$ are on the \emph{deterministic} matrices $\Sigma$ and $A$. In this paper, we focus on the random matrix
\begin{equation}
	\label{eq:modelY}
	Y = Y_N = U \Sigma V^*+A.
\end{equation}

We then introduce the limiting object of $Y$. Let $(\mathcal{A},\tau)$ be a $W^*$-probability space and $a,\sigma\in\mathcal{A}$ such that $A$ converges to $a$ and $\Sigma$ converges to $\sigma$ in $\ast$-distribution. Let $T\in \mathcal{A}$ be an $R$-diagonal freely independent from $a$ such that the law of $\vert T\vert$ is the same as that of $\sigma$; hence, if $u\in\mathcal{A}$ is a Haar unitary operator freely independent from $\sigma$, we have $T = u\sigma$ in $\ast$-distribution \cite{NicaSpeicher-Rdiag}. Set
\begin{equation}
	\label{eq:limity}
	y=T+a.
\end{equation}
By \cite[Theorem 9 on Page 105]{MingoSpeicherBook}, $Y$ converges in $\ast$-distribution to $y$; that is, for any polynomial $P$ in two noncommuting indeterminates,
\[\lim_{N\to\infty} \E\tr[P(Y,Y^*)] = \tau[P(y,y^*)].\]

We are interested in the limiting eigenvalue distribution of $Y$. Since $y$ is the limit in $\ast$-distribution of $Y$, a natural candidate of the limiting eigenvalue distribution of $Y$ is the \emph{Brown measure} of $y$, which we will define shortly. Note, however, that although $y$ is the limit in $\ast$-distribution of $Y$, it is not automatic that the Brown measure of $y$ is the limiting eigenvalue distribution of $Y$; indeed, there are examples of matrices whose limiting eigenvalue distribution does not converge to the Brown measure of their limit in $\ast$-distribution (see Section~\ref{sect:Jordan}). We now define the Brown measure introduced by Brown \cite{Brown1986}.

\begin{definition}[\cite{Brown1986}]
	\label{def:BrownDef}
	Let $x\in\mathcal{A}$. The function $h(\lambda)=\tau[\log\vert x-\lambda\vert]$ defined for $\lambda\in\C$ is a subharmonic function on $\C$. The Brown measure $\mu_x$ of $x$ is defined to be 
	\[\mu_x = \frac{1}{2\pi}\Delta h\]
	where the Laplacian is taken in distributional sense. The function $h$ turns out to be the logarithmic potential of the Brown measure of $x$ in the sense that
	\begin{equation}
		\label{eq:LogPotx}
		\tau[\log\vert x-\lambda\vert] = \int_{\C}\log\vert z - \lambda \vert d\mu_x(z).
	\end{equation}
	For simplicity, we call the function $h$ the logarithmic potential of $x$.
\end{definition}

The definition of the Brown measure generalizes the eigenvalue distribution of a square matrix. If we replace $x$ in Definition~\ref{def:BrownDef} by an $N\times N$ matrix $X$ and the tracial state $\tau$ by the normalized trace $\tr=(1/N)\Tr$ of matrices, then by Section 11.2 of \cite{MingoSpeicherBook},
\begin{equation}
	\label{eq:LogPotXDef}
	\tr[\log\vert X-\lambda\vert] = \frac{1}{N}\sum_{k=1}^N \log \vert \lambda_k - \lambda\vert
\end{equation}
where $\lambda_k$ are the eigenvalues of $X$, and $(1/2\pi)\Delta_\lambda\tr[\log\vert X-\lambda\vert]$ is the eigenvalue distribution of $X$. The function defined in \eqref{eq:LogPotXDef} is called the logarithmic potential of the empirical eigenvalue distribution of $X$.

We introduce the main theorems of this paper. For a probability measure $\mu$ on $\R$, we denote
\begin{equation}
	\label{eq:CauchyTransDef}
	G_\mu(z) = \int\frac{1}{z-t}d\mu(t),\quad z\in\C^+
\end{equation}
to be the Cauchy transform of $\mu$. We also write $\widetilde\mu$ to be the \emph{symmetrization} of $\mu$ defined by
\[\widetilde\mu(B) = \frac{1}{2}[\mu(B)+\mu(-B)]\]
for all Borel set $B\subset \R$. Given a Hermitian matrix $X$ or a Hermitian operator $x$, we write the law of $X$ or $x$ as $\mu_X$ or $\mu_x$ respectively. There is no notational ambiguity; the law $\mu_X$ or $\mu_x$ are the Brown measure of $X$ or $x$ respectively.

The Brown measure of $T+a$ is calculated in a recent work by Bercovici and the second author \cite{BercoviciZhong2022Rdiag}. Denote
\begin{equation}
\Omega(T, a)=\{ \lambda\in\mathbb{C}:
||(a-\lambda)^{-1}||_2 ||T||_2 >1 \quad\text{and}\quad
||a-\lambda||_2 ||T^{-1}||_2>1 \}.
\end{equation}
The Brown measure of $T+a$ is supported in the closure of $\Omega(T,a)$; see Theorem \ref{thm:main-1-BZhong2022} for a review.

 The main results of this paper are the following two theorems. The first one, proved in Theorem~\ref{thm:SpecificA}, concerns the case when $A$ is Hermitian for all $N$ or $A$ is unitary for all $N$. By putting $A = 0$, the theorem recovers the single ring theorem by Guionnet et al. \cite{GuionnetKZ-single-ring} with weaker conditions, removing the regularity assumption that there exists $\kappa_1,\kappa_2>0$ such that
 \[\vert\operatorname{Im}G_\Sigma(z)\vert\leq \kappa_1,\quad \operatorname{Im}z>N^{-\kappa_2}.\] 
 Basak and Dembo \cite[Proposition 1.3]{BasakDembo2013} relaxed this assumption. In any compact subset of the support of the single ring law, this regularity assumption has been removed; see \cite[Remark 1.14]{BaoES2019singlering}.  Our result completely removes this regularity assumption; in particular, it applies to the simplest case where $A=0$ and the eigenvalue distribution of $\Sigma$ converges to $\delta_1$. This completely solves the question in \cite[Remark 2]{GuionnetKZ-single-ring}.

 \begin{theorem}
	 Let $Y$ and $y$ be as in \eqref{eq:modelY} and \eqref{eq:limity}. If $A$ is Hermitian for all $N$ or if $A$ is unitary for all $N$, then the empirical eigenvalue distribution of $Y$ converges weakly to the Brown measure of $y$ in probability. 
 \end{theorem}
 
The following theorem, proved in Theorem~\ref{thm:deformedsinglering}, does not assume that $A$ is Hermitian or unitary. We put some relatively mild technical assumptions on $A$ and $\Sigma$. Assumption~(\ref{Assumption1}) in the following theorem is satisfied if $A$ is Hermitian for all $N$ or $A$ is unitary for all $N$; this assumption means the singular values of $A-\lambda$ do not concentrate around $0$ for most $\lambda$. 

\begin{theorem}
	\label{thm:maintheorem}
	Let $Y$ and $y$ be as in \eqref{eq:modelY} and \eqref{eq:limity}. Suppose that $\Sigma$ is invertible and $\|\Sigma^{-1}\|\leq N^{\alpha}$ for some $\alpha>0$ independent of $N$. Assume further that at least one of the following is true:
	\begin{enumerate}
		\item \label{Assumption1} there exists a Lebesgue measurable set $E\subset \Omega(T,a)^c$ that has Lebesgue measure $0$ with the following property: for any compact set $S\subset\Omega(T,a)^c\cap E^c$, there exist constants $\kappa_1, \kappa_2>0$ such that 
		 \begin{equation}
		   \label{eqn:condition1a}
		   \left\vert G_{\widetilde\mu_{\vert A-\lambda\vert}}(i\eta)\right\vert\leq\kappa_2
		 \end{equation}
		for all $\eta > N^{-\kappa_1}$ and $\lambda\in S$; or
		\item  there exist constants $\kappa_1, \kappa_2>0$ such that 
		\begin{equation}
		 \label{eqn:condition1b}
		 \left\vert G_{\widetilde{\mu}_{\Sigma}}(i\eta)\right\vert\leq\kappa_2
		\end{equation}
		for all $\eta > N^{-\kappa_1}$.
	\end{enumerate}
	Then the empirical eigenvalue distribution of $Y$ converges weakly to the Brown measure of $y$ in probability.
\end{theorem}

\begin{remark} [Discussions on methodologies]
The Hermitian reduction method is a well-known approach in non-Hermitian random matrix theory (see \cite{Bordenave-Chafai-circular, GuionnetKZ-single-ring, TaoVu2010-aop} and references therein). In this approach, a major difficulty is to estimate the least singular value. We apply the main result of Rudelson-Vershynin \cite{RudelsonVershynin2014} to get a least singular value estimate. On the technical level, the proof strategies for convergence of the deformed single ring model follow closely with preceeding works on the single ring model \cite{GuionnetKZ-single-ring}, \cite{BenaychGeorges2017} and \cite{BaoES2019singlering}. The recent results in free probability \cite{BercoviciZhong2022Rdiag, BelinschiBercoviciHo2022} provide additional tools and guide us to extend arguments for the single ring model to the deformed model in this paper. 
\end{remark}

\begin{remark}
Assumptions similar to Condition \eqref{eqn:condition1a} and Condition \eqref{eqn:condition1b} appear in the paper \cite{GuionnetKZ-single-ring} on the single ring matrix model. We briefly explain these two conditions in the next paragraph.

Condition \eqref{eqn:condition1a} quantifies the behavior of the (small) singular values of $A-\lambda$ for $\lambda\not\in(\overline{\Omega(T,a)})^c$ in terms of its Cauchy transform on the positive imaginary axis. This condition is inspired by the behavior of the limiting operator $a-\lambda$ as follows. From \cite[Lemma 4.8]{BercoviciZhong2022Rdiag}, it is known that the support of $\mu_a$ is contained in the closure of $\Omega(T,a)$. In general, $\text{supp}(\mu_a)\subset\text{spec}(a)$ (see \cite{HaagerupLarsen2000, HaagerupSchultz2007}). There are examples that $\text{supp}(\mu_a)=\text{spec}(a)$, where one can deduce that for any compact $S\subset (\overline{ \Omega(T,a)})^c$, $0\notin \text{supp}(\mu_{\vert a-\lambda\vert})$ for any $\lambda\in S$. In these examples, $\vert G_{\widetilde\mu_{\vert a-\lambda\vert}}(i\eta)\vert$ is uniformly bounded for $\lambda\in S$. In the progress of proving the eigenvalue convergence, Condition \eqref{eqn:condition1a} controls the Cauchy transform of $A-\lambda$ in a way similar to the preceding discussion about $a-\lambda$; it gives a uniform bound for $\vert G_{\widetilde\mu_{\vert A-\lambda\vert}}(i\eta)\vert$ for $\lambda\in S$ and $\eta>N^{-\kappa_1}$. We emphasize that, Condition \eqref{eqn:condition1a} is a condition on the random matrix $A$ but not the limiting operator $a$. We are not making the assumption that $\text{supp}(\mu_a)=\text{spec}(a)$. Similarly, Condition \eqref{eqn:condition1b} quantifies the invertibility (the small eigenvalues) of the matrix $\Sigma$. 
\end{remark}

\begin{remark}
	In \cite{GuionnetKZ-single-ring} and \cite{BaoErdosSchnelli2017}, the authors consider another random matrix model $O_1 \Sigma O_2$ where $\Sigma$ is a nonnegative matrix, $O_1$ and $O_2$ are independent Haar orthogonal matrices. It is thus a natural question to ask whether we can replace the Haar unitary matrices $U$ and $V$ in \eqref{eq:modelY} by Haar orthogonal matrices. In the present paper, we use results by Benaych-Georges \cite{BenaychGeorges2017} on approximate subordination functions, which are only proved in the case of Haar unitary matrices. We believe the conclusion of Theorem~\ref{thm:maintheorem} holds in the case of Haar orthogonal matrices; however, we do not prove this case in our paper.
\end{remark}

The empirical eigenvalue distribution of $Y$ is the distributional Laplacian of \[(1/2\pi)\tr[\log\vert Y-\lambda\vert]\] with respect to $\lambda$. To show the weak convergence of the empirical eigenvalue distribution of $Y$ to the Brown measure of $y$, we need to show that, for any test function $f\in C_c^\infty(\C)$, 
\[\int \Delta f(\lambda)\tr[\log\vert Y-\lambda\vert]d^2\lambda \to \int\Delta f(\lambda)\tau[\log\vert y-\lambda\vert]d^2\lambda\] 
in probability. Note that $\tr[\log\vert Y-\lambda\vert]$ is the average of the logarithm of the singular values of $Y-\lambda$. Since the logarithm is unbounded around $0$, we need to estimate the least singular value of $Y-\lambda$ for all $\lambda\in\C$ in order to control $\tr[\log\vert Y-\lambda\vert]$. In all of our applications, the least singular value of $Y-\lambda$ can be estimated using Theorem~\ref{thm:lsv}.

We denote
\begin{equation}\nonumber
\begin{aligned}
D(T,a)=\{ \lambda\in\mathbb{C}: 0<f_{|T+a-\lambda|}(0)<\infty \},
	\end{aligned}
	\end{equation}
	where $f_{|T+a-z|}$ denotes the density function of the absolutely continuous part of 
	$|T+a-z|$. One can show that $D(T,a)\subset \Omega(T,a)$. See Section \ref{section:2.2.BrownMeasureLimit} for an alternative definition of $D(T,a)$. 
 The following is a simplified version of Theorem \ref{thm:localconv}, which also states the rate of the convergence. This is a deformed local single ring theorem in Theorem~\ref{thm:localconv}, which concerns the local behavior of the eigenvalues of $Y$ in the bulk $D(T,a)$. 
\begin{theorem}
	Let $Y$ and $y$ be as in \eqref{eq:modelY} and \eqref{eq:limity}. Suppose that $\Sigma$ is invertible and $\|\Sigma^{-1}\|\leq N^{\alpha}$ for some $\alpha>0$ independent of $N$. Let $a_N\in \mathcal{A}$ be such that $a_N$ has the \emph{same} $\ast$-distribution as $A_N$. Also let $T_N\in \mathcal{A}$ be an $R$-diagonal operator such that $\mu_{\vert T_N\vert} = \mu_{\Sigma_N}$ and $T_N$ is freely independent from $a_N$. Write
	\[y_N = a_N+T_N.\]
	For any compact set $K\subset D(T,a)$, $\alpha\in (0,1/2)$, $w_0\in K$ and any smooth function $f:\C\to\R$ supported in the disk centered at $0$ of radius $R$, we define a function $f_{w_0}$ depending on $N$ by
	\[f_{w_0}(w) = N^{2\alpha} f(N^\alpha(w-w_0)).\]
	Denote by $\lambda_1,\ldots,\lambda_N$ the eigenvalues of $Y$. We have 
	\[
		\frac{1}{N}\sum_{k=1}^N f_{w_0}(\lambda_k) - \int_\C f_{w_0}(w)\,d\mu_{y_N}(w)\to 0
	\]
	in probability, where $\mu_{y_N}$ is the Brown measure of $y_N$. This convergence is uniform in $f$ with $\|\Delta f\|_{L^1(\C)}\leq 1$ and in $w_0\in K$, for all large enough $N$ depending on $K$, $R$, $M$, $a$ and $T$. (The constant $M$ is defined in \eqref{eq:Mmeaning}.)
\end{theorem}

\section{Preliminaries and auxiliary results}

\subsection{Free additive convolution}
Given a probability measure $\mu$ on $\mathbb{R}$, its Cauchy transform $G_\mu$ is defined on the complex upper half-plane $\mathbb{C}^+$ as in \eqref{eq:CauchyTransDef}. Let $F_\mu: \mathbb{C}^+\rightarrow \mathbb{C}^+$ be an analytic map defined as
\[
F_\mu(z)=\frac{1}{G_\mu(z)}, \qquad z\in\mathbb{C}^+.
\] 
Given two probability measures $\mu_1, \mu_2$ on $\mathbb{R}$, it is known that there exist a pair of analytic functions $\omega_1, \omega_2: \mathbb{C}^+\rightarrow\mathbb{C}^+$, such that, for all $z\in\mathbb{C}^+$, we have 
\begin{equation} 
\label{eqn:subordination-scalar}
F_{\mu_1\boxplus\mu_2}(z)=F_{\mu_1}(\omega_1(z))=F_{\mu_2}(\omega_2(z))
=\omega_1(z)+\omega_2(z)-z. 
\end{equation}
The free additive convolution obeys various regularity properties. Suppose neither of $\mu_1,\mu_2$ is a point mass. Then $\mu_1\boxplus\mu_2$ has no singular continuous part, and the absolutely continuous part of $\mu_1\boxplus \mu_2$ is always nonzero, and its density is analytic whenever positive and finite. 
See \cite{Belinschi2008, survey-2013, BercoviciVoiculescu1998} for more details.

Recall that the symmetrization $\widetilde\mu$ of a probability measure $\mu$ on $\R$ is defined by 
\[
 \widetilde{\mu}(B)\deq\frac{1}{2}\big[ \mu(B) + \mu(-B)\big]
\] 
for any Borel set $B\subset\R$. Set
\begin{align}
\label{defn:mu-sigma-xi}
\mu_{\sigma,\xi}\deq\widetilde{\mu}_{\sigma}\boxplus  \widetilde{\mu}_{\xi}, 
\end{align}
where $\boxplus$ denotes the free additive convolution of probability measures on $\R$. For any symmetric probability measure $\mu$ on $\mathbb{R}$, observe that, by the symmetry, 
\[
    G_\mu(i\eta)=\int_\mathbb{R}\frac{1}{i\eta-x}d\mu(x)=-i\eta\int_\mathbb{R}\frac{1}{\eta^2+x^2}d\mu(x), \quad\eta>0.
\]
Hence, $F_\mu(i\eta)\in i(0,\infty)$. 
We can deduce that the subordination functions $\omega_1, \omega_2$ corresponding to the free convolution of two symmetric probability measures $\widetilde{\mu}_\sigma,\widetilde{\mu}_\xi$ satisfy
\[\omega_1(i\eta), \omega_2(i\eta)\in i(0,\infty)\]
for all $\eta>0$. See \cite[Proposition 3.1]{BercoviciZhong2022Rdiag}. 

\subsection{The Brown measure and limit distribution}
\label{section:2.2.BrownMeasureLimit}
Recall that the pair $(\mathcal{A}, \tau)$ denotes a $W^*$-probability space. 
Given a sequence of random matrices $\{X_N\}$, we say that $X_N$ converges in $\ast$-distribution (or $\ast$-moments) to $x\in(\mathcal{A},\tau)$ if
\[
\lim_{N\rightarrow \infty} \E\tr(P(X_N, X_N^*))=\tau( P(x, x^*)), 
\]
for any polynomial $P$ in two noncommuting indeterminates with complex coefficients, where $\tr$ means the normalized trace $(1/N)\Tr$.

Recall that $U=U_N$ and $V=V_N$ are two independent Haar random unitary matrices, and $\Sigma=\Sigma_N$ is a sequence of $N\times N$ deterministic nonnegative definite diagonal matrices, and $A=A_N$ is a sequence of $N\times N$ deterministic matrices. Denote the singular values of $\Sigma$ by $\sigma_1, \cdots,\sigma_N$ (which are also the eigenvalues of $\Sigma$), and the empirical distribution of the singular values of $\Sigma$ by 
\[
\mu_\Sigma =\frac{1}{N}\sum_{i=1}^N\delta_{\sigma_i}.
\]
We assume there is a nonnegative operator $\sigma\in\mathcal{A}$ such that
\[
\mu_\Sigma\rightarrow \mu_\sigma.
\]
Recall that $A$ and $U\Sigma V^*$ converge in $*$-moments to $a$ and $T$ respectively, where $T$ is $R$-diagonal and $a$ is freely independent from $T$. 

By \eqref{eqn:subordination-scalar}, the free convolution of $\mu_1$ and $\mu_2$, where
\[\mu_1=\widetilde{\mu}_{|a-\lambda|} \textrm{ and } \mu_2=\widetilde{\mu}_{|T|}=\widetilde\mu_{\sigma}.\] 
has a pair of subordination functions depend on $\lambda$. We consider them as two-variable functions and write them as
\[\omega_1(\lambda,\cdot) \textrm{ and } \omega_2(\lambda,\cdot).\]
Bercovici and the second author \cite{BercoviciZhong2022Rdiag} find that the Brown measure of $y$ can be calculated using the values of $\omega_1(\lambda,0)$ and $\omega_2(\lambda,0)$ for each $\lambda\in\C$. They also prove (\cite[Lemma 3.3]{BercoviciZhong2022Rdiag}) that, for each $\lambda\in\C$, at least one of $\omega_1(\lambda,0)$ and $\omega_2(\lambda,0)$ is finite. Section~\ref{sect.approx.subordination} studies the approximate subordination functions for matrices, and the observation that at least one of $\omega_1(\lambda,0)$ and $\omega_2(\lambda,0)$ is finite provide huge conveniences to prove that these approximate subordination functions converge to $\omega_1$ and $\omega_2$. On the other hand, Belinschi, Bercovici and the first author \cite{BelinschiBercoviciHo2022} recently use the Denjoy--Wolff point theory to study continuity of subordination functions. This continuity also allows us to control the behavior of the approximate subordination functions when we take the limit as the size of the matrices to infinity.

The following proposition proves that $\omega_1$ and $\omega_2$ have continuous extensions. The proof uses the fact that the values of the subordination functions at any $z\in\C^+$ in \eqref{eqn:subordination-scalar} can be identified as a unique particular solution of a fixed point equation of a holomorphic self-map on $\C^+$. This fixed point is called the Denjoy--Wolff point. See \cite[Chapter 5]{Shapiro1993composition} for an introduction to the Denjoy--Wolff theory (see also Section 1 in \cite{BelinschiBercoviciHo2022}). The continuous extensions of $\omega_1$ and $\omega_2$ relies on the continuity of Denjoy--Wolff points \cite{Heins1941, BelinschiBercoviciHo2022}.

\begin{proposition}
	\label{prop:contomega}
	The subordination functions $\omega_1$ and $\omega_2$ have continuous extensions on $\C\times (\C^+\cup\R)$.
In particular, the continuous extensions of $\omega_j$ satisfy 
\[\omega_j(\lambda,0) = \lim_{\eta\to 0} \omega_j(\lambda,i\eta),\quad\lambda\in\C.\]
\end{proposition}
\begin{proof}
	Consider the family of holomorphic functions parametrized by $\lambda$ and $z$ defined by
	\[\varphi_{\lambda,z}(w) = F_{\mu_2}(F_{\mu_1}(w)-w+z)-(F_{\mu_1}(w)-w+z)+z, \quad w\in\C^+,\]
	which is a self-map on the upper half plane. Since neither $\mu_1$ nor $\mu_2$ is a single delta mass, for any $\lambda\in\C$, $z\in\C^+\cup\R$, $\omega_1(\lambda,z)$ is the Denjoy--Wolff point of $\varphi_{\lambda,z}$ (\cite[Corollary 3.5]{BelinschiBercoviciHo2022}). 
	
	It is clear that $\mu_1=\widetilde{\mu}_{|a-\lambda|}$ is continuous with respect to $\lambda$ in the space of probability measures on $\mathbb{R}$ equipped with the weak topology. Therefore, if $\lambda_n\to\lambda$ in $\C$ and $z_n\to z$ in $\C^+\cup\R$, we have $\varphi_{\lambda_n,z_n}\to \varphi_{\lambda, z}$ pointwise. By Theorem 1.1 of \cite{BelinschiBercoviciHo2022}, $\omega_1(\lambda_n,z_n)\to \omega_1(\lambda,z)$. This shows the continuity of $\omega_1$. The proof of the continuity of $\omega_2$ is similar.
\end{proof}

\begin{proposition}
	\label{prop:suborduniformconv}
	Let $a_N,\sigma_N\in\mathcal{A}$ be such that $\sigma_N\geq 0$, $a_N\to a$ and $\sigma_N\to\sigma$ in $\ast$-distribution. Given $\lambda\in\mathbb{C}$, consider the free additive convolution of $\widetilde{\mu}_{|a_N-\lambda|}$ and $\widetilde{\mu}_{\sigma_N}$ and let $\omega_1^{(N)}(\lambda,\cdot)$ and $\omega_2^{(N)}(\lambda,\cdot)$ be the corresponding subordination functions. Then for any compact sets $K\subset\C$ and $K'\subset\C^+\cup\R$, if both $\omega_1(\lambda,z)$ and $\omega_2(\lambda,z)$ are uniformly bounded for all $\lambda\in K$ and $z\in K'$, then
	\[(\omega_1^{(N)}(\lambda,z),\omega_2^{(N)}(\lambda,z))\to (\omega_1(\lambda,z),\omega_2(\lambda,z))\] 
	uniformly in $\lambda\in K$ and $z\in K'$. 
\end{proposition}
\begin{proof}
	If the conclusion does not hold, there exist $\varepsilon_0>0$ and sequences $\lambda_N\in K$ and $z_N\in K'$ such that 
	\begin{equation}
		\label{eq:subordNUnif}
		\left\vert \omega_1^{(N)}(\lambda_N,z_N)-\omega_1(\lambda_N,z_N)\right\vert\geq \varepsilon_0
	\end{equation}
	or
	\[\left\vert \omega_2^{(N)}(\lambda_N,z_N)-\omega_2(\lambda_N,z_N)\right\vert\geq \varepsilon_0.\]
	Without loss of generality, we assume the former holds. By extracting subsequences of $(\lambda_N)$ and $(z_N)$ if necessary, we assume $\lambda_N\to\lambda$ and $z_N\to z$ for some $\lambda\in K$ and $z\in K'$.

	Write $\mu_{1,N} = \widetilde{\mu}_{|a_N-\lambda_N|}$ and $\mu_{2,N} = \widetilde{\mu}_{\sigma_N}$. Consider the holomorphic function
	\[\varphi_{N}(w) = F_{\mu_{2,N}}(F_{\mu_{1,N}}(w)-w+z_N)-(F_{\mu_{1,N}}(w)-w+z_N)+z_N, \quad w\in\C^+\]
	which is a self-map on the upper half plane. Define $\varphi_{\lambda, z}$ as in the proof of Proposition~\ref{prop:contomega}. Then $\varphi_N\to \varphi_{\lambda, z}$ pointwise. But then by Theorem 1.1 of \cite{BelinschiBercoviciHo2022}, $\omega_1^{(N)}(\lambda_N, z_N)\to \omega_1(\lambda, z_N)$, contradicting~\eqref{eq:subordNUnif}.
\end{proof}

Now we discuss the support of the Brown measure of $T+a$. Recall that
\begin{equation}
\label{def:support-Omega-set}
\Omega(T, a)=\{ \lambda\in\mathbb{C}:
||(a-\lambda)^{-1}||_2 ||T||_2 >1 \quad\text{and}\quad
||a-\lambda||_2 ||T^{-1}||_2>1 \}.
\end{equation}
We set 
\begin{equation}
\label{def:singular-set-S}
S(T, a)=\{\lambda\in\mathbb{C}: \mu_{1}(\{0\}) + \mu_{2}(\{0\})\geq 1 \}.
\end{equation}
Since neither $\mu_1$ nor $\mu_2$ is a single delta mass at zero, it follows that $\lambda\in S(T, a)$ implies that $0$ is an eigenvalue of $|T|$ and $|a-\lambda|$. We see that $\lambda$ satisfies the defining conditions for $\Omega(T, a)$. Hence, $S(T, a)\subset \Omega(T, a)$.

\begin{theorem}\cite[Theorem 4.1 and Proposition 4.11]{BercoviciZhong2022Rdiag}
	\label{thm:main-1-BZhong2022}
	For any $\lambda\in\C$ and $\eta>0$, $\omega_j(\lambda, i\eta)$ is a purely imaginary number ($j=1,2$).  The set $S(T, a)$ consists of finitely many elements and we have
	\begin{equation}
	\begin{aligned}
	\Omega(T,a)\backslash S(T, a)&= \{ \lambda\in\mathbb{C}:\vert\omega_1(\lambda,0)\vert\in (0,\infty)  \}\\
	&= \{ \lambda\in\mathbb{C}: \vert\omega_2(\lambda,0)\vert\in (0,\infty)\}\\
	&=\{ \lambda\in\mathbb{C}: 0<f_{\mu_1\boxplus\mu_2}(0)<\infty \},
	\end{aligned}
	\end{equation}
	where $f_{\mu_1\boxplus\mu_2}$ denotes the density function of the absolutely continuous part of 
	$\mu_1\boxplus\mu_2$.
	Moreover, 
	\[
	S(T, a)=\{ \lambda\in\mathbb{C}: \omega_1(\lambda,0)= \omega_2(\lambda,0)=0\},
	\]
	and
	\[
	   \mathbb{C}\backslash \Omega(T,a)=\{\lambda\in\mathbb{C}: \text{exactly one of}\, \,\,\omega_1(\lambda,0), \omega_2(\lambda,0)\,\,\, \text{is infinity}\}.
	\]
	
	The Brown measure of $T+a$ is supported in the closure of $\Omega(T, a)$ and is absolutely continuous in $D(T, a)$, where 
	\[
	D(T, a)=\Omega(T,a)\backslash S(T, a).
	\]
\end{theorem}

\begin{definition}
	\label{def:Depsilon}
	For any $\varepsilon>0$, we denote 
	\begin{equation*}
	D^{(\varepsilon)}(T,a)= \{ \lambda\in\mathbb{C}: \vert\omega_1(\lambda, 0)\vert, \vert\omega_2(\lambda,0)\vert \in (\varepsilon,1/\varepsilon)  \}.
	\end{equation*}
\end{definition}

We then have 
\[
D(T,a)=\bigcup_{n=1}^\infty D^{(1/n)}(T,a).
\]

\begin{proposition}
	\label{prop:boundedCauchy}
	Let $K\subset\mathbb{C}\backslash\Omega(T,a)$ be a compact set. Consider the free additive convolution of $\mu_1$ and $\mu_2$. Then there exists a constant $C>0$ such that 
	\[G_{\mu_1\boxplus \mu_2}(i \eta)\leq C\]
	for all $\eta\geq 0$ and $\lambda\in K$.
\end{proposition}
\begin{proof}
	By the subordination relation~\eqref{eqn:subordination-scalar} and Proposition~\ref{prop:contomega}, $(\lambda, \eta)\mapsto G_{\mu_1\boxplus \mu_2}(i \eta)$ is continuous in $\lambda\in K$ and $\eta\geq 0$.  By Theorem \ref{thm:main-1-BZhong2022}, exactly one of $\omega_1(\lambda,0), \omega_2(\lambda,0)$ is infinity when $\lambda\in\mathbb{C}\backslash \Omega(T, a)$, thus $G_{\mu\boxplus\mu_2}(0)=0$. This shows that $G_{\mu_1\boxplus \mu_2}(i \eta)$ is bounded for $(\lambda,\eta) \in K \times (0,1]$. The proposition then follows from the observation that $\vert G_{\mu_1\boxplus \mu_2}(i \eta)\vert\leq 1$ if $\eta\geq 1$.
\end{proof}

\subsection{A result about the random matrix model}
For any $\lambda\in\mathbb{C}$, we have to study the limit of the singular value distribution of 
the random matrix $Y-\lambda $. 
Since $U$ and $V$ are independent Haar-distributed unitary matrices, both $Y-\lambda=U \Sigma V^*+A-\lambda$
and $U\Sigma V^*+ \vert A-\lambda \vert$ have the same singular values.
To this end, we consider the $N\times N$ random matrix of the form
\begin{equation}
\label{eqn:defn-X-N}
X=X_N=U\Sigma V^*+ \Xi
\end{equation}
where $U=U_N$ and $V=V_N$ are two independent Haar random unitary matrices, $\Sigma=\Sigma_N$ and $\Xi$ are sequences of deterministic matrices.
The general result about $X$ can be applied to $Y-\lambda$ by choosing $\Xi = \vert A - \lambda\vert$.

Without loss of generality, we may assume that $\Xi$ is a diagonal matrix
\begin{equation*}
\Xi=\text{diag}(\xi_1,\cdots, \xi_N),
\end{equation*}
where $\xi_i\in\mathbb{C}$ for all $i=1,\cdots, N$. 
For $\lambda$ in any compact set $K\subset \C$, $\Vert A-\lambda\Vert$ is bounded uniformly. Thus we may assume there exists some constant $C$ independent of $N$ such that
\begin{equation} 
\label{C-constant-inequality}
\Vert\Sigma\Vert, \Vert\Xi\Vert\leq C.
\end{equation}
This constant $C$ depends on the compact set $K$, which will be taken to be the support of a fixed test function when we prove the convergence of the empirical eigenvalue distribution of $Y$. So we surpress the dependence of $K$ in the notation of the constant $C$.

Denote the empirical distribution of the singular values of $\Xi$ by 
\[
 \mu_\Xi=\frac{1}{N}\sum_{i=1}^N\delta_{|\xi_i|},
\]
and assume that there is a $\xi\in\mathcal{A}$ such that
\[
\mu_\Xi\rightarrow \mu_\xi.
\]
Throughout this paper, we always assume that $\mu_\sigma$ and $\mu_\xi$ are not a single point mass at zero. 

A general approach for a non-Hermitian random matrix is to consider its Hermitian reduction. 
We consider Hermitian random matrices $H$ defined by 
\begin{equation}
\label{def:matrix-H}
H=\begin{bmatrix}
  U & 0\\
  0 & V
\end{bmatrix}\begin{bmatrix}
 0 &\Sigma\\
 \Sigma^* & 0
\end{bmatrix}
\begin{bmatrix}
U^* & 0\\
0 & V^*
\end{bmatrix} +\begin{bmatrix}
0 & \Xi\\
\Xi^* & 0
\end{bmatrix}.
\end{equation}
The eigenvalues of the matrix $\begin{bmatrix}
 0 &\Sigma\\
 \Sigma^* & 0
\end{bmatrix}$ are exactly $\{|\sigma_i|, -|\sigma_i|: i=1,\cdots, N\}$, where $|\sigma_i|$ are singular values of $\Sigma$. Similarly, the eigenvalues of the matrix $\begin{bmatrix}
0 & \Xi\\
\Xi^* & 0
\end{bmatrix}$ are exactly $\{|\xi_i|, -|\xi_i|: i=1,\cdots, N\}$, where $|\xi_i|$ are singular values of 
$\Xi$.

By the asymptotic freeness result \cite{SpeicherNicaBook, Voiculescu1991}, the random matrices $(U\Sigma V^*+\Xi)^*(U\Sigma V^*+\Xi)$ converges in $*$-moment to $(T+\xi)^*(T+\xi)$, where $T$ is an $R$-diagonal operator and $\xi$ is $*$-free from $T$. Hence, $\mu_{|U\Sigma V^*+\Xi|}\rightarrow \mu_{|T+\xi|}$ weakly. Let $u$ be a Haar unitary operator that is $*$-free from $\{T, \xi\}$, then $|T+u\xi|$ and $T+\xi$ have the same $*$-distribution. It is known that $u\xi$ is $R$-diagonal. Hence, $T+u\xi$ is a sum of two $R$-diagonal operators. By \cite[Proposition 3.5]{HaagerupSchultz2007} (see also \cite{Nica-Speicher-1998duke}), we have 
\[
  \widetilde{\mu}_{|T+u\xi|}
  =\widetilde{\mu}_\sigma\boxplus\widetilde{\mu}_\xi=\mu_{\sigma, \xi}.
\]
We conclude that the empirical eigenvalue distribution of $H$ converges weakly 
to $\mu_{\sigma,\xi}$. In \cite{BaoES2019singlering}, a qualitative version of this convergence was obtained. Given an interval $I\subset \mathbb{R}$ and $0\leq a\leq b$, we denote
\begin{equation}
S_I(a,b)=\{ z=x+i \eta\in\mathbb{C}^+: x\in I, a \leq \eta\leq b \}.
\end{equation}

Following \cite{BaoES2019singlering}, we use the following definition taken from \cite{ErdosKYAHP2013}. 
\begin{definition}[Stochastic domination]
Let $\mathcal{X}=\mathcal{X}^{(N)}$, $\mathcal{Y}=\mathcal{Y}^{(N)}$ be two sequence of nonnegative random variables. We say that $\mathcal{Y}$ stochastically dominates $\mathcal{X}$ if, for all (small) $\varepsilon>0$ and (large) $D>0$,
\begin{equation}
\mathbb{P}( \mathcal{X}^{(N)}>N^\varepsilon \mathcal{Y}^{(N)} )\leq N^{-D},
\end{equation}
for sufficiently large $N\geq N_0(\varepsilon, D)$, and we write $\mathcal{X}\prec \mathcal{Y}$. When $\mathcal{X}^{(N)}$ and $\mathcal{Y}^{(N)}$ depend on a parameter $w\in W$. We say $\mathcal{X}(w)\prec \mathcal{Y}(w)$ uniformly in $w$ if $N_0(\varepsilon, D)$ can be chosen independent from $w$. 
\end{definition}

\begin{definition}
	\label{def:bulk}
	For two Borel probability measures $\nu_1, \nu_2$ on $\mathbb{R}$ and neither of $\nu_1,\nu_2$ is a point mass, denote by $f_{\nu_1\boxplus\nu_2}$ the density function of the absolutely continuous part of $\nu_1\boxplus\nu_2$. The bulk of $\nu_1\boxplus\nu_2$ is defined as
	\begin{equation}
	\mathcal{B}_{\nu_1\boxplus\nu_2}=\{ x\in\mathbb{R}: 0<f_{\nu_1\boxplus\nu_2}(x)<\infty, \nu_1\boxplus\nu_2(\{x\})=0  \}.
	\end{equation}
\end{definition}

We denote by $d_{\mathrm{L}}(\mu, \nu)$ the L\'{e}vy distance of two probability measures $\mu$ and $\nu$ on $\mathbb{R}$.
\begin{theorem}[\cite{BaoES2019singlering}]
\label{thm:main-lemma-Bao-etal}
Let $\mu_\sigma, \mu_\xi$ be two compactly supported probability measures on $[0,\infty)$ such that neither $\widetilde{\mu}_{\sigma}$ nor $\widetilde{\mu}_{\xi}$ is a single point mass and at least one of them is supported at more than two points. Fix some $L>0$ and let $I$ be any compact subinterval of the bullk $\mathcal{B}_{\widetilde{\mu}_{\sigma}\boxplus \widetilde{\mu}_{\xi}}$. Then there exists a constant $b_0>0$ and $N_0\in\mathbb{N}$, depending on $\mu_\sigma,\mu_\xi, I$ and the constant $C$ in \eqref{C-constant-inequality}, such that whenever
\[
\sup_{N\geq N_0}\left( d_{\mathrm{
L}}(\mu_\Sigma,\mu_\sigma) +d_{\mathrm{L}}(\mu_\Xi, \mu_\xi) \right) \leq 2b, 
\]
for some $b\leq b_0$, then 
\begin{equation}
\label{eqn:difference-Cauchy-transform-fixed-z}
|G_H(z)-G_{\widetilde\mu_{\Sigma}\boxplus \widetilde\mu_{\Xi}}(z)| 
\prec\frac{1}{N\eta (1+\eta)} 
\end{equation}
holds uniformly on $S_I(0, N^L)$, for $N$ sufficiently large depending on $\mu_\sigma, \mu_\xi, I, L$ and the constant $C$ given in \eqref{C-constant-inequality}. 

Moreover, there exists a constant $\eta_M\geq 1$, independent of $N$, such that 
\eqref{eqn:difference-Cauchy-transform-fixed-z} holds uniformly on $S_I(\eta_M, N^L)$, for 
any compact interval $I\subset \mathbb{R}$, for $N$ sufficiently large depending on $\mu_\sigma, \mu_\xi, I, L$ and the constant $C$ in \eqref{C-constant-inequality}.
\end{theorem}

\subsection{Approximate subordination functions}
\label{sect.approx.subordination}
Free additive convolution can be studied using the subordination functions. When we work with a sum of two random matrices that are asymptotically free, there is also a pair of approximate subordination functions \cite{Kargin2015}. To study the singular value distribution of the sum, we use the results in \cite{BenaychGeorges2017}. In the following, we introduce the approximate subordination functions in \cite{BenaychGeorges2017} and state a direct consequence of \cite[Theorem 1.5]{BenaychGeorges2017}. 

Let $\Sigma$ and $\Xi$ be deterministic $N\times N$ matrices such that there is a constant $M$ independent of $N$ such that 
\begin{equation}
	\label{eq:matrixnormbound}
	\|\Sigma\|,\|\Xi\|\leq M.
\end{equation}
 Let $U$, $V$ be independent $N\times N$ Haar-distributed unitary matrices and $\widetilde{\Sigma} = U \Sigma V^*$. Define 
\[{\bf A} = \begin{pmatrix} 0& \Xi\\
	\Xi^* & 0
\end{pmatrix};\quad{\bf B} = \begin{pmatrix} 0& \widetilde\Sigma\\
	\widetilde{\Sigma}^* & 0
\end{pmatrix};\quad {\bf H} = \begin{pmatrix} 0& \Xi+\widetilde{\Sigma}\\
	(\Xi+\widetilde{\Sigma})^* & 0
\end{pmatrix}.\]
The matrix ${\bf H}$ is the same as the $H$ in~\eqref{def:matrix-H}. We use bold font here just in this section for consistency of the use of bold fonts of ${\bf A}$ and ${\bf B}$. Following \cite[Eq.(67)]{BenaychGeorges2017}, we define a complex-valued function $\omega_{\bf A}$ by
\begin{equation}
	\label{eq:SBformula}
	\omega_{\bf A}(z) = z+\frac{1}{\E G_{\bf H}(z)}\left(\frac{1}{2N}\E \Tr\left[\left(z - {\bf H}\right)^{-1}{\bf B}\right]\right)
\end{equation}
and similarly define $\omega_{\bf B}$ by changing ${\bf B}$ to ${\bf A}$.

\begin{theorem}
	\label{thm:ApproxSubordination}
	For some functions $r_{\bf A}(z)$ and $r_{\bf B}(z)$ depending on $z\in\C^+$, we have
	\begin{align*}
		\E G_{\bf H}(z) &= \E G_{\bf A}(\omega_{\bf A}(z))+r_{\bf A}(z),\\
		\E G_{\bf H}(z) &= \E G_{\bf B}(\omega_{\bf B}(z))+r_{\bf B}(z).
	\end{align*} 
	Write $\eta = \operatorname{Im}(z)$. There exists a constant $C>0$ that only depends on $M$ in \eqref{eq:matrixnormbound} such that 
	\begin{enumerate}
		\item if $N\eta^5\geq C$, then $\operatorname{Im} \omega_{\bf A}(z), \operatorname{Im}\omega_{\bf B}(z)\geq \eta-\frac{C}{N\eta^7}$; and
		\item if $N\eta^8\geq C$, then $\vert r_{\bf A}(z)\vert, \vert r_{\bf B}(z)\vert\leq \frac{C}{N\eta^6}$.
	\end{enumerate}
\end{theorem}
\begin{proof}
	Let $R_{\bf A}$ and $R_{\bf B}$ as in (24) and (25) in \cite{BenaychGeorges2017}. Let $r_{\bf A}(z) = \frac{1}{2N}\Tr[R_{\bf A}(z)]$ and $r_{\bf B}(z) = \frac{1}{2N}\Tr[R_{\bf B}(z)]$. Meanwhile, the function $S_{\bf B}$ in \cite[Eq.(67)]{BenaychGeorges2017} is related to $\omega_{\bf A}$ by
	\[\omega_{\bf A}(z) = z+S_{\bf B}(z)\]
	and similarly $\omega_{\bf B}(z) = z+S_{\bf A}(z)$. The conclusion of the theorem follows directly from \cite[Theorem~1.5]{BenaychGeorges2017}.
\end{proof}

We only need the values of the functions $\omega_{\bf A}(z)$ and $\omega_{\bf B}(z)$ for purely imaginary $z$.
\begin{lemma}
	\label{lem:SImag}
	For any $\eta>0$, $\omega_{\bf A}(i\eta)$ and $\omega_{\bf B}(i\eta)$ are purely imaginary.
\end{lemma}
\begin{proof}
Note that $\omega_{\bf A}$ is analytic in $\mathbb{C}^+$. We will show that $\omega_{\bf A}(z)=-\overline{\omega_{\bf A}(-\overline{z})}$ for $z\in\mathbb{C}^+$ with $|z|$ large enough. 

Write $S= \Xi+\widetilde{\Sigma}$. Since $\bf H$ has a symmetric distribution, we have 
\begin{equation}\label{eqn:216-inproof}
     -\overline{\E[G_{\bf H}(-\overline{z})]}=\E[G_{\bf H}(z)]
\end{equation}
for all $z\in\mathbb{C}^+$. In particular, 
$\E[G_{\bf H}(i\eta)]$ is purely imaginary. By the formula \eqref{eq:SBformula} of $\omega_{\bf A}$, it suffices to show that 
\begin{equation}
	\label{eq:TrfB}
	\frac{1}{2N}\E\Tr\left[\left(z-{\bf H}\right)^{-1}{\bf B}\right] = -\overline{\frac{1}{2N}\E\Tr\left[\left(-\bar z-{\bf H}\right)^{-1}{\bf B}\right]}.
\end{equation}
for all large enough $\vert z\vert$.

We expand \eqref{eq:TrfB} into a power series
\begin{equation}
	\label{eq:ExpandSeries}
	\sum_{n=0}^\infty\frac{1}{z^{n+1}}\E\frac{1}{2N}\Tr \left[\begin{pmatrix} 0 &S\\
	S^* & 0
\end{pmatrix}^n\begin{pmatrix} 0 &\widetilde\Sigma\\
	\widetilde\Sigma^* & 0
\end{pmatrix}  \right].
\end{equation}
It is straightforward to see (for example, by mathematical induction) that, if $n =2m$,
\[\begin{pmatrix} 0 &S\\
	S^* & 0
\end{pmatrix}^n = \begin{pmatrix}
	(SS^*)^m & 0\\
	0& (S^*S)^m
\end{pmatrix}\]
and if $n = 2m+1$,
\[\begin{pmatrix} 0 &S\\
	S^* & 0
\end{pmatrix}^n = \begin{pmatrix}
	0& (SS^*)^mS\\
	(S^*S)^mS^* & 0
\end{pmatrix}.\]
Thus,
\[\frac{1}{2N}\Tr \left[\begin{pmatrix} 0 &S\\
	S^* & 0
\end{pmatrix}^n\begin{pmatrix} 0 &\widetilde\Sigma\\
	\widetilde\Sigma^* & 0
\end{pmatrix}  \right] = \begin{cases}
	\operatorname{Re}\tr[(SS^*)^mS\widetilde\Sigma^*], \quad&\textrm{if $n=2m+1$}\\
	0, & \textrm{if $n=2m$}
\end{cases}.\]
This shows only the odd terms in \eqref{eq:ExpandSeries} survive. By combining \eqref{eqn:216-inproof}, it follows that 
$\omega_{\bf A}(z)=-\overline{\omega_{\bf A}(-\overline{z})}$ for $z\in\mathbb{C}^+$ with $|z|$ large enough.
Since $z\mapsto -\overline{\omega_{\bf A}(-\overline{z})}$ defines an analytic function on $\C^+$, it follows that $\omega_{\bf A}(z)=-\overline{\omega_{\bf A}(-\overline{z})}$ for all  $z\in\mathbb{C}^+$. Hence $\omega_{\bf A}(i\eta)$ is purely imaginary for all $\eta>0$. 
\end{proof}

\section{Least singular value estimate}
\label{sect:LSV}

To estimate the logarithmic potential of the random matrix
\[Y = U\Sigma V^*+A,\]
we need to estimate the least singular value of $Y-\lambda$, denoted by $s_{\mathrm{min}}(Y-\lambda)$, for $\lambda\in\C$. Since
\begin{equation}
	\label{eq:sminU}
	s_{\mathrm{min}}(Y-\lambda) = s_{\mathrm{min}}(U\Sigma+(A-\lambda)V) = s_{\mathrm{min}}(U^*(A-\lambda)+\Sigma V^*)
\end{equation}
and $U^*$ is also Haar-distributed, it suffices to study the least singular value $s_{\mathrm{min}}(Z)$ of the random matrix 
\[Z = U\Gamma  + D\]
where $U$ is a Haar-distributed unitary matrix, $D$ is an arbitrary matrix, and $\Gamma$ is a invertible matrix with $\Vert \Gamma^{-1}\Vert\leq N^{\alpha}$ for some $\alpha> 0$ independent of $N$. The main theorem of this section is to estimate the least singular value $s_{\mathrm{min}}(Z)$; the estimate only depends on $\alpha$ and is independent of $D$.

The following lemma is probably well-known, but we do not know a reference. We include a proof for this lemma.
\begin{lemma}
	\label{lem:prodlsv}
	For any $N\times N$ matrices $A_1, A_2$,
	\[s_{\mathrm{min}}(A_1A_2)\geq s_{\mathrm{min}}(A_1)s_{\mathrm{min}}(A_2).\]
\end{lemma}
\begin{proof}
	If one of $A_1$ or $A_2$ is singular, then the inequality is trivial. We assume both $A_1$ and $A_2$ are invertible. By the min-max theorem,
	\begin{align*}
		s_{\mathrm{min}}(A_1A_2)& = \inf_{\Vert x\Vert_2=1} \Vert A_1 A_2 x\Vert\\
		&= \inf_{\Vert x\Vert_2=1} \left\Vert A_1 \frac{A_2 x}{\Vert A_2 x\Vert}\right\Vert\Vert A_2x\Vert\\
		&\geq s_{\mathrm{min}}(A_1)s_{\mathrm{min}}(A_2).
	\end{align*}
	Thus the lemma is established.
\end{proof}

The following result is a simple application of a result of Rudelson--Vershynin \cite{RudelsonVershynin2014} regarding the least singular value of $U\Sigma V^*$. 
\begin{theorem}
	\label{thm:lsv}
	Suppose that $\Gamma$ is invertible and fix $\alpha>0$ such that $\Vert\Gamma^{-1}\Vert\leq N^\alpha$. There exist positive constants $c$ and $c'$ depending on $\alpha$ but independent of $D$ and $N$ such that
	\[\mathbb{P}(s_{\mathrm{min}}(U\Gamma  + D)<t)\leq t^cN^{c'},\quad t>0.\]
\end{theorem}
\begin{proof}
	First note that $s_{\mathrm{min}}(\Gamma) = 1/\Vert \Gamma^{-1}\Vert$. By writing 
	\[U\Gamma+D = (U+D\Gamma^{-1})\Gamma,\]
	Lemma~\ref{lem:prodlsv} shows that
	\begin{align*}
		s_{\mathrm{min}}(U\Gamma+D)& = s_{\mathrm{min}}((U+D\Gamma^{-1})\Gamma)\\
		&\geq s_{\mathrm{min}}(U+D\Gamma^{-1})s_{\mathrm{min}}(\Gamma)\\
		&\geq N^{-\alpha} s_{\mathrm{min}}(U+D\Gamma^{-1}).
	\end{align*}
	Therefore,
	\begin{equation}
		\label{eq:Y-lambdalsv}
		\mathbb{P}(s_{\mathrm{min}}(U\Gamma+D)<t)\leq \mathbb{P}(s_{\mathrm{min}}(U+D\Gamma^{-1})<N^{\alpha}t).
	\end{equation}
	Theorem 1.1 of \cite{RudelsonVershynin2014} states that 
	\begin{equation}
		\label{eq:Rudelson}
		\mathbb{P}(s_{\mathrm{min}}(U+D)<t)\leq t^{c} N^{C}
	\end{equation}
	for some positive constants $c$ and $C$ independent of the matrix $D$. We then apply \eqref{eq:Rudelson} with $D\Gamma^{-1}$ in place of $D$ to \eqref{eq:Y-lambdalsv} and conclude
	\[\mathbb{P}(s_{\mathrm{min}}(U\Gamma+D)<t)\leq t^{c}N^{C+c\alpha}.\]
	The conclusion of our theorem then holds with $c' = C+c\alpha$.
\end{proof}

\section{The deformed local single ring theorem}
\subsection{Convergence of Cauchy transforms in the bulk}
Consider the random matrix model $Y$ as in \eqref{eq:modelY} and operator $y$ in \eqref{eq:limity}. To study the Cauchy transform of $Y$, we apply the random matrix model $X$ in \eqref{eqn:defn-X-N} with $\Xi = \vert A-\lambda\vert$. In this section, we study the rate of the convergence of the Cauchy transform of $Y$. For any $\lambda$, we write (following the Girko's trick)
 \begin{equation}
 \label{defn-H-omega}
    H^\lambda= \begin{pmatrix}
         0 & Y-\lambda
         \\
        Y^*-\overline{\lambda} & 0
    \end{pmatrix}.
 \end{equation}
 The eigenvalues of $H^\lambda$ are exactly $\{ z_i^\lambda, - z_i^\lambda, i=1, \cdots, N \}$,
 where $z_i^\lambda$ are singular values of $Y-\lambda$. Denote by $G^\lambda$ the Cauchy transform of $H^\lambda$. Then,
 \[
    G^\lambda(z)=\frac{1}{2N}\left( \frac{1}{z-z_i^\lambda} +\frac{1}{z+z_i^\lambda} \right).
 \] 
 We denote 
 \begin{equation}
     G_{\Sigma, |A-\lambda|}=G_{\mu_{\Sigma, |A-\lambda|}}
 \end{equation}
 where $\mu_{\Sigma, |A-\lambda|}$ is the free convolution of $\widetilde{\mu}_\Sigma$ and $\widetilde{\mu}_{\vert A-\lambda\vert}$, consistent to the notation \eqref{defn:mu-sigma-xi}.

By using the Brown measure results for $T+a$ described in Section \ref{section:2.2.BrownMeasureLimit}, we can then extend the approach in \cite[Theorem 2.5]{BaoES2019singlering} (see also \cite{BenaychGeorges2017}) to the deformed model as follows. 
\begin{theorem} 
	\label{thm:HlambdaCauchy}
	Recall that $M$ is defined in~\eqref{eq:Mmeaning} and $D^{\varepsilon}(T,a)$ is defined in Definition~\ref{def:Depsilon}. For any $L_0>0$, there exists $N_0$ depending on $\varepsilon, T, a$ and $M$, the estimate 
\begin{equation}
  \label{eqn:estimate-main}
  \sup_{\lambda \in D^{(\varepsilon)}(T,a)} |G^\lambda (i\eta) -G_{\Sigma, |A-\lambda|}(i\eta)| \prec \frac{1}{N\eta},
\end{equation}
holds uniformly in $\eta>N^{-L_0}$ for all  $N\geq N_0$. 
\end{theorem}
\begin{proof}
For any $\lambda\in\mathbb{C}$, let 
\[
    A-\lambda=U_1^* \Xi^\lambda V_1
\]
be the singular value decomposition of $A-\lambda$, where $U_1, V_1$ are unitary matrices and $\Xi^\lambda=\text{diag}(\xi_1^\lambda, \cdots, \xi_N^\lambda)$ is a diagonal matrix consisting of the singular values $\xi_j^\lambda$ of $A-\lambda$ on the diagonal. Then, the matrix $U\Sigma V^*+A-\lambda$ has the same singular values as $U_1 U \Sigma V^*V_1 +\Xi^\lambda$. Then $U_1U$ and $V_1V$ are again two independent Haar random unitary matrices. Hence, $U\Sigma V^*+A-\lambda$ has the same singular values as $U\Sigma V^*+\Xi^\lambda$.
Moreover, since the matrices $A-\lambda$ converge in $*$-moments to $a-\lambda$, it follows that 
\[
   \widetilde{\mu}_{\vert A-\lambda\vert}\longrightarrow \widetilde{\mu}_{\vert a-\lambda\vert}
\]
weakly. Recall that our random matrix model assumes $\mu_\Sigma\to\mu_\sigma$ weakly.

If $|z|\gg 1$, we note that for a symmetric probability measure $\mu$ on $\mathbb{R}$, we can expand
\[
   G_\mu(z)=\int_{\mathbb{R}}\frac{1}{z-u}d\mu(u)=\sum_{k=0}^\infty \frac{m_{2k}(\mu)}{z^{2k+1}},
\]
where $m_n(\mu)=\int_\mathbb{R} u^n d\mu(u)$ is the $n$-th moment of $\mu$. Consequently, for any $\lambda \in D^{(\varepsilon)}(T, a)$, we have 
\begin{equation}
  \label{eqn:4.4-in-proof}
    G^\lambda(i\eta)= \frac{1}{i\eta}+O\left( \frac{1}{\eta^3} \right), \qquad G_{\Sigma, |A-\lambda|}(i\eta)= \frac{1}{i\eta}+O\left( \frac{1}{\eta^3} \right)
\end{equation}
provided that $\eta>N^L$ and $L\geq L_0$ is sufficiently large. The above approximation is uniform for all $\lambda \in D^{(\varepsilon)}(T, a)$. Hence, \eqref{eqn:estimate-main} holds uniformly in $\lambda$ for $\eta>N^L$. In the rest of the proof, we show that \eqref{eqn:estimate-main} also holds uniformly in $\lambda$ for $N^{-L_0}\leq\eta\leq N^L$.

Recall that we write $\mu_1=\widetilde{\mu}_{|a-\lambda|}$ and $\mu_2=\widetilde{\mu}_{|T|} = \widetilde{\mu}_\sigma$. There are subordination functions in the sense of \eqref{eqn:subordination-scalar} such that $\omega_1(\lambda,z)+\omega_2(\lambda,z)=F_{\mu_1\boxplus\mu_2}(z)+z$. Hence,
for $\lambda\in  D^{(\varepsilon)}(T,a)$, 
\[
  f_{\mu_1\boxplus\mu_2}(0)=-\frac{1}{\pi}\lim_{\eta\rightarrow 0}\operatorname{Im} G_{\mu_1\boxplus\mu_2}(i\eta)
   \in \left( \frac{\varepsilon}{2\pi}, \frac{1}{2\pi\varepsilon} \right).
\]
Hence by Definition~\ref{def:bulk}, $0 \in \mathcal{B}_{\mu_1\boxplus\mu_2}$. By applying Theorem \ref{thm:main-lemma-Bao-etal} for the choice $\mathcal{I}=\{0\}$, \eqref{eqn:estimate-main} holds uniformly for $0\leq \eta\leq N^L$ for $\lambda
\in D^{(\varepsilon)}(T, a)$ fixed. 

Fix some large $L$ so that \eqref{eqn:4.4-in-proof} holds for any $\lambda \in D^{(\varepsilon)}(T, a)$.  We next adapt the approach in the proof of \cite[Theorem 2.5]{BaoES2019singlering} for $N^{-L_0}\leq \eta\leq N^L$. 
By the definition of stochastic domination and \eqref{eqn:estimate-main} for any $\lambda \in D^{(\varepsilon)}(T,a)$ fixed, the number of elements in the lattice $D^{(\varepsilon)}(T,a)\cap \{N^{-L_1}+i N^{-L_1}\}$ is of order $N^{2L_1}$, hence we have 
\[
  \max_{\lambda \in D^{(\varepsilon)}(T,a)\cap N^{-L_1}\{\mathbb{Z}\times i\mathbb{Z} \}}
     |G^\lambda (i\eta) -G_{\Sigma, |A-\lambda|}(i\eta)| \prec \frac{1}{N\eta}.
\]
That is, the inequality \eqref{eqn:estimate-main} holds for $\lambda$ at the lattice points in $D^{(\varepsilon)}(T,a)$, where the coordinates of these points are multiple of $N^{-L_1}$.

It remains to prove some Lipschitz type inequality with respect to the parameter $\lambda$. Let $L_1$ be a large positive number such that $L_1\geq 2L$. We claim that, for $\lambda_1
, \lambda_2\in D^{(\varepsilon)}(T,a)$ with $|\lambda_1-\lambda_2|\leq N^{-L_1}$, we have 
\begin{equation}
  \label{ineq-1-in-proof}
    |G^{\lambda_1} (i\eta)-G^{\lambda_2} (i\eta)|\prec \frac{1}{N\eta}
\end{equation}
and 
\begin{equation}
  \label{ineqn-2-in-proof}
  |G_{\Sigma, |A-\lambda_1|}(i\eta)-G_{\Sigma, |A-\lambda_2|}(i\eta)|\prec \frac{1}{N\eta}
\end{equation}
uniformly in  $N^{-L}\leq N^{-L_0}\leq \eta \leq N^L$. Recall the definition of the random matrix $H^\lambda$ \eqref{defn-H-omega}. By the resolvent identity, if $\eta \geq N^{-L_0}\geq N^{-L}$, we have 
\begin{align*}
    |G^{\lambda_1} (i\eta)-G^{\lambda_2} (i\eta)|
      &=\left|\frac{1}{2N}\operatorname{Tr}(i\eta - H^{\lambda_1})^{-1} - \frac{1}{2N}\operatorname{Tr}(i\eta - H^{\lambda_2})^{-1}\right|\\
       &\leq \frac{|\lambda_1-\lambda_2|}{2N} |\operatorname{Tr} \left( (i\eta - H^{\lambda_1})^{-1}  (i\eta - H^{\lambda_2})^{-1} \right)|\\
        &\leq |\lambda_1-\lambda_2|\cdot \Vert ( H^{\lambda_1}-i\eta)^{-1}\Vert\cdot \Vert ( H^{\lambda_2}-i\eta)^{-1}\Vert\\
       &\leq \frac{|\lambda_1-\lambda_2|}{\eta^2} \leq \frac{1}{\eta} N^{-L_1+L}\leq \frac{1}{N\eta}.
\end{align*}
This establishes \eqref{ineq-1-in-proof}.

We now apply \cite[Equation (2.20)]{BaoES2016-jfa} to get 
\[
  |G_{\Sigma, |A-\lambda_1|}(i\eta)-G_{\Sigma, |A-\lambda_2|}(i\eta)|
   \leq \frac{C}{\eta}\left( 1+\frac{1}{\eta} \right) d_{\mathrm{L}}( \widetilde{\mu}_{|A-\lambda_1|}, \widetilde{\mu}_{|A-\lambda_2|}),
\]
for all $\eta>0$ and some constant $C$ independent from $\eta$. Note that $\widetilde{\mu}_{|A-\lambda|}$ has the same distribution as the eigenvalue distribution of the matrix $\begin{bmatrix}
  0 & A-\lambda\\
  A^*-\overline{\lambda} &0
\end{bmatrix}$.
By some standard inequality for spectral measure as in \cite[Proposition 1.6 (iii)]{Fack1982}, we have 
\begin{align*}
  d_{\mathrm{L}}( \widetilde{\mu}_{|A-\lambda_1|}, \widetilde{\mu}_{|A-\lambda_2|})
   &\leq \left \lVert\begin{bmatrix}
  0 & A-\lambda_1\\
  A^*-\overline{\lambda_1} &0
\end{bmatrix}-\begin{bmatrix}
  0 & A-\lambda_2\\
  A^*-\overline{\lambda_2} &0
\end{bmatrix} \right \rVert\\
     &\leq |\lambda_1-\lambda_2|\leq \frac{1}{N^{L_1}}\leq \frac{1}{N}.
\end{align*}
This proves \eqref{ineqn-2-in-proof}. We conclude that \eqref{eqn:estimate-main} holds uniformly in $\lambda$ for $N^{-L_0}\leq \eta\leq N^L$. Since we already prove that \eqref{eqn:estimate-main} holds uniformly in $\lambda$ for $\eta> N^L$ in the first two paragraph of the proof, the theorem is established.
\end{proof}

\subsection{Proof of the local convergence}
In this section, we prove a deformed local single ring theorem for the random matrix model $Y$ in \eqref{eq:modelY} following the approaches in \cite{BenaychGeorges2017} and \cite{BaoES2019singlering}. The random matrix $Y$ converges in $\ast$-distribution to $y$ in \eqref{eq:limity}. Recall that we assume $\|A\|, \|\Sigma\|\leq M$ for all $N$. We denote
\[H^\lambda = \begin{pmatrix}
	0 & Y-\lambda\\
	(Y-\lambda)^* & 0
\end{pmatrix}.\]

\begin{theorem}
	\label{thm:localconv}
	Consider the random matrix model $Y$ and the operator $y\in\mathcal{A}$ as in \eqref{eq:modelY} and \eqref{eq:limity}. Suppose that $\Sigma$ is invertible and $\|\Sigma^{-1}\|\leq N^{\alpha}$ for some $\alpha>0$ independent of $N$. Write $\lambda_1,\ldots,\lambda_N$ be the eigenvalues of the random matrix $Y$. 
	
	Let $R>0$. For any compact set $K\subset D(T,a)$, $\beta\in (0,1/2)$, $w_0\in K$ and any $f:\C\to\R$ be a smooth function supported in the disk centered at $0$ of radius $R$, define the function by
	\[f_{w_0}(w) = N^{2\beta} f(N^\beta(w-w_0)).\]
	Let $a_N, \sigma_N\in \mathcal{A}$ be such that $a_N$ and $\sigma_N$ have the \emph{same} $\ast$-distribution as $A_N$ and $\Sigma_N$. Also let $T_N\in \mathcal{A}$ be an $R$-diagonal operator such that $\mu_{\vert T_N\vert} = \mu_{\Sigma_N}$. Write
	\[y_N = a_N+T_N.\]
	Then
	\begin{equation}
		\label{eq:localsinglering}
		\frac{1}{N}\sum_{k=1}^N f_{w_0}(\lambda_k) - \int_\C f_{w_0}(w)\,d\mu_{y_N}(w)\prec N^{-1+2\beta}\|\Delta f\|_{L^1(\C)},
	\end{equation}
	 where $\mu_{y_N}$ is the Brown measure of $y_N$. This convergence is uniform in $f$ and in $w_0\in K$, for all large enough $N$ depending on $K$, $R$, $M$, $a$ and $T$. (The constant $M$ is defined in \eqref{eq:Mmeaning}.)
\end{theorem}

We need several lemmas before we can prove Theorem~\ref{thm:localconv}.

\begin{lemma}
	\label{lem:boundedintG}
	For any $\delta>0$, there exists $L_1>1$ such that for any $\lambda\in D^{(\delta)}(T,a)$,
	\begin{equation}
		\label{eq:boundedintG}
		\left\vert\int_0^{N^{-L_1}}G_{\Sigma,\vert A-\lambda\vert}(i\eta)\,d\eta\right\vert\leq \frac{1}{N}
	\end{equation}
	for all large enough $N$.
\end{lemma}
\begin{proof}
	Our strategy is to apply Proposition~\ref{prop:suborduniformconv} with that $a_N$ has the same distribution as $A=A_N$ and $\sigma_N$ has the same distribution as $\Sigma=\Sigma_N$. We use the notation $\omega_1^{(N)}(\lambda,\cdot)$ and $\omega_2^{(N)}(\lambda,\cdot)$ in Proposition~\ref{prop:suborduniformconv} to denote the subordination functions corresponding to the free convolution of $\widetilde{\mu}_{\vert A-\lambda\vert}$ and $\widetilde{\mu}_{\vert \Sigma\vert}$.
	
	By the subordination relation~\eqref{eqn:subordination-scalar},
	\[G_{\Sigma,\vert A-\lambda\vert}(i\eta) = \frac{1}{\omega_1^{(N)}(\lambda,i\eta)+\omega_1^{(N)}(\lambda,i\eta)-i\eta}.\]
	By Definition~\ref{def:Depsilon}, the subordination functions $\omega_1(\lambda,\cdot)$ and $\omega_2(\lambda,\cdot)$ corresponding to the free convolution of $\widetilde{\mu}_{\vert a-\lambda\vert}$ and $\widetilde{\mu}_{\sigma}$ satisfy $\omega_1(\lambda,0),\omega_2(\lambda,0)\in(\delta/2,2/\delta)$ for all $\lambda\in D^{(\delta/2)}(T,a)$. We then apply Proposition~\ref{prop:suborduniformconv} to see that $G_{\Sigma,\vert A-\lambda\vert}(i\eta)$ is uniformly bounded for $\lambda\in \overline{D^{(\delta)}(T,a)}\subset D^{(\delta/2)}(T,a)$ and $\eta\in[0,1]$. Thus, we can choose $L_1$ large enough such that \eqref{eq:boundedintG} holds.
\end{proof}

\begin{lemma}
	\label{lem:singularvalueG}
	For any $\alpha>0$, there exist positive constants $c$ and $\widetilde{c}$ such that for any $L_1>0$, 
	\[\E\left\vert\int_0^{N^{-L_1}}G^\lambda(i\eta)\,d\eta\right\vert\leq N^{-cL_1/2+\widetilde{c}}\]
	provided that $\Vert (A-\lambda)^{-1}\Vert \leq N^{\alpha}$ or $\Vert \Sigma^{-1}\Vert \leq N^{\alpha}$. We have used the convention that if a matrix $\Gamma$ is not invertible, set $\Vert \Gamma^{-1}\Vert = \infty$.
\end{lemma}
We will use this estimate in the proof of Theorem~\ref{thm:localconv} with the condition $\Vert \Sigma^{-1}\Vert\leq N^\alpha$. We will apply this lemma again in Theorems~\ref{thm:deformedsinglering} and~\ref{thm:SpecificA}. In the proof of Theorem~\ref{thm:SpecificA}, we use the other condition $\Vert (A-\lambda)^{-1}\Vert \leq N^{\alpha}$.
\begin{proof}
We follow the approach in \cite[Proposition 14]{GuionnetKZ-single-ring} (see also \cite{BenaychGeorges2017} and \cite[Theorem 1.8]{BaoES2019singlering}).
	Denote $s_{\mathrm{min}}^\lambda$ to be the least singular value of $Y-\lambda$; that is, the least nonnegative eiganvalue of $H^\lambda$. We first note that if $\Vert (A-\lambda)^{-1}\Vert \leq N^{\alpha}$ or $\Vert \Sigma^{-1}\Vert \leq N^{\alpha}$, by~\eqref{eq:sminU}, we can apply the least singular value estimate in Theorem~\ref{thm:lsv} to $Y-\lambda$.
	
	Since $G^\lambda$ is the Cauchy transform of the eigenvalues of $H^\lambda$,
\begin{align}
	\E\left\vert \int_0^{N^{-L_1}}G^\lambda(i\eta)\,d\eta\right\vert &\leq \E \int_0^{N^{-L_1}}\frac{\eta}{(s_{\mathrm{min}}^\lambda)^2+\eta^2}\,d\eta\nonumber\\
	&=\frac{1}{2}\E[\log(1+(N^{L_1}\cdot s_{\mathrm{min}}^\lambda)^{-2})]\nonumber\\
	&=\frac{1}{2}\int_0^\infty\bbP(\log(1+(N^{L_1}\cdot s_{\mathrm{min}}^\lambda)^{-2})\geq s)ds\nonumber\\
	&=\frac{1}{2}\int_0^\infty\bbP\left(s_{\mathrm{min}}^\lambda\leq N^{-L_1}\frac{1}{\sqrt{e^s-1}}\right)ds\label{eq:Cauchy.Singularest}.
\end{align}
We decompose the integral into three parts $\int_0^{N^{-L_1}}+\int_{N^{-L_1}}^1+\int_1^\infty$. For the first integral, it is straightforward to see that
\[\int_0^{N^{-L_1}}\bbP\left(s_{\mathrm{min}}^\lambda\leq N^{-L_1}\frac{1}{\sqrt{e^s-1}}\right)ds\leq N^{-L_1}.\]
For the second integral, we estimate that $\frac{1}{\sqrt{e^s-1}}\leq \frac{1}{\sqrt{s}}\leq N^{L_1/2}$ for all $N^{-L_1}\leq s\leq 1$. By Theorem~\ref{thm:lsv}, there exist positive constants $c$ and $c'$
\[\int_{N^{-L_1}}^1\bbP\left(s_{\mathrm{min}}^\lambda\leq N^{-L_1}\frac{1}{\sqrt{e^s-1}}\right)ds\leq\int_{N^{-L_1}}^1\bbP(s_{\mathrm{min}}^\lambda\leq N^{-L_1/2})\leq N^{-cL_1/2+c'}.\]
Finally, for the third integral, using $e^{s}-1>\frac{1}{2}e^s$ for all $s\geq 1$, we have
\begin{align*}
	\int_{1}^\infty\bbP\left(s_{\mathrm{min}}^\lambda\leq N^{-L_1}\frac{1}{\sqrt{e^s-1}}\right)ds&\leq\int_1^\infty\bbP(s_{\mathrm{min}}^\lambda\leq N^{-L_1}\sqrt{2}e^{-s/2})\,ds\\
	&\leq \int_1^\infty 2^{c/2}N^{-cL_1+c'}e^{-cs/2}\,ds\\
	&=\frac{2^{1+c/2}e^{-c/2}}{c}N^{-cL_1+c'}
\end{align*}
for some positive constants $c$ and $c'$. Put these three estimates of integrals to \eqref{eq:Cauchy.Singularest}, we have, for some positive constants $c$ and $\widetilde{c}$,
the conclusion holds.
\end{proof}

We are now ready to prove Theorem~\ref{thm:localconv} following the approach in \cite{BenaychGeorges2017} and \cite{BaoES2019singlering}.

\begin{proof}[Proof of Theorem~\ref{thm:localconv}]
We do a change of variable $\lambda = N^\beta(w-w_0)$ and write
\begin{align}
	&\quad\frac{1}{N}\sum_{k=1}^N f_{w_0}(\lambda_k) - \int_\C f_{w_0}(w)\,d\mu_{y_N}(w)\nonumber \\
	&= \frac{N^{2\beta}}{2\pi}\int_{\C}(\Delta f)(\lambda)\left(\frac{1}{2N}\Tr\log\vert H^w\vert - \int_\R \log\vert u\vert d\mu_{\Sigma,\vert A-w\vert}(u)\right)d^2\lambda.\label{eq:CompareLog}
\end{align}
By modifying the approach in \cite{BaoES2019singlering}, we will prove that 
\begin{equation}
 \label{eqn:term1-local}
	 \int_{\C}(\Delta f)(\lambda)\left(\frac{1}{2N}\Tr\log\vert H^w\vert - \int_\R \log\vert u\vert d\mu_{\Sigma,\vert A-w\vert}(u)\right)d^2\lambda \prec \frac{\|\Delta f\|_{L^1(\C)}}{N}
\end{equation}
uniformly in $w\in\overline W$ where $W$ is any fixed neighborhood of $K$ such that $\overline W\subset D(T,a)$ is compact. Note that for all $N$ large enough, $f_{w_0}$ is supported in $\overline{W}$ for all $w_0\in K$.

We now estimate \eqref{eqn:term1-local}. For any $L>0$ and $w$, we follow an observation from \cite{TaoVu2010-aop} and write 
\begin{equation}
\begin{aligned}
\frac{1}{2N}\Tr\log\vert H^w\vert& = \frac{1}{2N}\Tr\log \vert(H^w-iN^L)\vert+\operatorname{Im}\int_0^{N^L} G^w(i\eta)\,d\eta\\
		\int_\R\log\vert u\vert\,d\mu_{\Sigma,\vert A-w\vert}(u)& = \int_\R \log\vert u-iN^L\vert\,d\mu_{\Sigma,\vert A-w\vert}(u) + \operatorname{Im}\int_0^{N^L} G_{\Sigma,\vert A-w\vert}(i\eta)\,d\eta.\label{eq:IntCauchy}
\end{aligned}
\end{equation}
It is clear that there is a constant $C$ such that $\|H^w\|\leq C$ for all $w$ in any ball of finite radius. Since the support of $f_{w_0}$ lies in a ball of radius $RN^{-\alpha}$, we can then choose $L$ large enough so that
\begin{equation}
	\label{eq:LogImag}
	\left\vert\frac{1}{2N}\Tr\log\vert H^w-iN^L\vert - \int_\R \log\vert u-iN^L\vert d\mu_{\Sigma,\vert A-w\vert}(u)\right\vert \ll \frac{1}{N}.
\end{equation}
Thus, it remains to estimate the second terms in \eqref{eq:IntCauchy}; we will show
\begin{equation}
	\label{eq:secondterm}
	\left\vert \int_\C(\Delta f)(\lambda)\left(\int_0^{N^L}(G^w(i\eta)-G_{\Sigma,\vert A-w\vert}(i\eta))d\eta\right)d^2\lambda\right\vert \prec \frac{\|\Delta f\|_{L^1(\C)}}{N}.
\end{equation}
We will choose large enough constants $L_1$ and $L$ and decompose the integral with respect to $\eta$ into
\[\int_0^{N^L} = \int_0^{N^{-L_1}}+\int_{N^{-L_1}}^{N^L}.\]

We first analyze the integral $\int_{N^{-L_1}}^{N^L}$. For the fixed neighborhood $W\subset D(T,a)$ of $K$, recall that there exists $\varepsilon>0$ such that $\overline W\subset D^{(\varepsilon)}(T,a)$. If $N$ is large enough, the support of $f_{w_0}$ must also lie in $\overline W$ for all $w_0\in K$. We then derive from Theorem~\ref{thm:HlambdaCauchy} that the stochastic domination
\begin{align*}
	\left\vert \int_\C(\Delta f)(\lambda)\left(\int_{N^{-L_1}}^{N^L}(G^w(i\eta)-G_{\Sigma,\vert A-w\vert}(i\eta))d\eta\right)d^2\lambda\right\vert &\prec \int_\C\vert (\Delta f)(\lambda)\vert \left(\int_{N^{-L_1}}^{N^L}\frac{1}{N\eta}d\eta\right)d^2\lambda\\
	&\prec\frac{\|\Delta f\|_{L^1(\C)}}{N}.
\end{align*} 

Next, we analyze the integral $\int_0^{N^{-L_1}}$. Our strategy is to choose $L_1$ large enough so that
\begin{equation}
	\label{eq:free.integral}
	\left\vert\int_\C(\Delta f)(\lambda)\left(\int_0^{N^{-L_1}} G_{\Sigma,\vert A-w\vert}(i\eta)\,d\eta\right)d^2\lambda\right\vert\leq \frac{\|\Delta f\|_{L^1(\C)}}{N}
\end{equation}
and
\begin{equation}
	\label{eq:N.integral}
	\bbP\left(\int_\C(\Delta f)(\lambda)\left(\int_0^{N^{-L_1}}G^w(i\eta)\,d\eta\right)d^2\lambda\geq \frac{\|\Delta f\|_{L^1(\C)}}{N}\right)\ll \frac{1}{N}.
\end{equation}
Since $\overline{W}\subset D^{(\varepsilon)}(T,a)$, Lemma~\ref{lem:boundedintG} shows that we can choose $L_1$ large enough such that \eqref{eq:free.integral} holds uniformly for test function $f$ supported in the disk centered at $0$ of radius $R$.

To show \eqref{eq:N.integral}, we use Markov's inequality to deduce
\begin{align}
	&\quad\bbP\left(\left\vert\int_\C(\Delta f)(\lambda)\left(\int_0^{N^{-L_1}}G^w(i\eta)\,d\eta\right)d^2\lambda\right\vert>\frac{\|\Delta f\|_{L^1(\C)}}{N}\right)\nonumber\\
	&\leq \frac{N}{\|\Delta f\|_{L^1(\C)}}\E\left\vert\int_\C(\Delta f)(\lambda)\left(\int_0^{N^{-L_1}}G^w(i\eta)\,d\eta\right)d^2\lambda\right\vert.\label{eq:MarkovCauchy}
\end{align}
By Lemma~\ref{lem:singularvalueG}, there exist positive constants $c$ and $\widetilde c$ such that we can estimate \eqref{eq:MarkovCauchy} by 
\begin{align*}
	\E\left\vert\int_\C(\Delta f)(\lambda)\left(\int_0^{N^{-L_1}}G^w(i\eta)\,d\eta\right)d^2\lambda\right\vert &\leq \int_\C\vert(\Delta f)(\lambda)\vert \E\left\vert\int_0^{N^{-L_1}}G^w(i\eta)\,d\eta\right\vert\,d^2\lambda\\
	&\leq \|\Delta f\|_{L^1(\C)}N^{-cL_1/2+\widetilde{c}}.
\end{align*}
Therefore, using \eqref{eq:MarkovCauchy}, we can choose $L_1$ large enough such that \eqref{eq:N.integral} also holds.

We have proved \eqref{eq:secondterm} by decomposing the integral into $\int_0^{N^{-L_1}}$ and $\int_{-N_{L_1}}^{N^L}$. Combining with the estimate \eqref{eq:LogImag}, using \eqref{eq:CompareLog} and \eqref{eq:IntCauchy}, we conclude that \eqref{eq:localsinglering} holds.
\end{proof}

\section{The deformed single ring theorem}
\label{sect:singlering}
Recall that $Y$ in \eqref{eq:modelY} converges in $\ast$-distribution to $y=T+a$ in \eqref{eq:limity}. 
In this section, we prove two deformed single ring theorems for the random matrix model $Y$ in \eqref{eq:modelY}. Section~\ref{sect:general-singlering} proves a deformed single ring theorem under some technical assumptions on $A$ and $\Sigma$; the matrix $A$ is \emph{not} assumed to be a normal matrix. Section~\ref{sect:hermitian-singlering} proves another version of deformed single ring theorem, showing that if $A$ is Hermitian or unitary, then the empirical eigenvalue distribution of $Y$ converges to the Brown measure of $y$ without additional assumption.

\subsection{The general case}
\label{sect:general-singlering}
Recall that we assume $\Vert A\Vert,\Vert\Sigma\Vert\leq M$ for all $N$. The set $\Omega(T,a)$ is defined in \eqref{def:support-Omega-set}. We denote
\[H^\lambda = \begin{pmatrix}
	0 & Y-\lambda\\
	(Y-\lambda)^* & 0
\end{pmatrix}\]
and write $m$ to be the Lebesgue measure on $\C$.

\begin{theorem}
	\label{thm:deformedsinglering}
	Consider the random matrix model $Y$ and the operator $y\in\mathcal{A}$ as in \eqref{eq:modelY} and \eqref{eq:limity}. Suppose that $\Sigma$ is invertible and $\|\Sigma^{-1}\|\leq N^{\alpha}$ for some $\alpha>0$ independent of $N$. Assume at least one of the following is true:
	\begin{enumerate}
		\item \label{StrongAssumption1} there exists a Lebesgue measurable set $E\subset \Omega(T,a)^c$ with $m(E)=0$ satisfying the following property: for any compact set $S\subset\Omega(T,a)^c\cap E^c$, there exist constants $\kappa_1, \kappa_2>0$ such that 
		\[\left\vert G_{\widetilde{\mu}_{\vert A-\lambda\vert}}(i\eta)\right\vert\leq\kappa_2\] 
		for all $\eta > N^{-\kappa_1}$ and $\lambda\in S$; or
		\item \label{StrongAssumption2} there exist constants $\kappa_1, \kappa_2>0$ such that 
		\[\left\vert G_{\widetilde{\mu}_{\Sigma}}(i\eta)\right\vert\leq\kappa_2\] 
		for all $\eta > N^{-\kappa_1}$.
	\end{enumerate}
	Then the empirical eigenvalue distribution of $Y$ converges weakly to the Brown measure of $y$ in probability.
\end{theorem}

Now we proceed to prove Theorem~\ref{thm:deformedsinglering}. We follow the approaches in \cite{GuionnetKZ-single-ring}, \cite{BenaychGeorges2017} and \cite{BaoES2019singlering} by cooperating with the free probability results in Section \ref{section:2.2.BrownMeasureLimit}. 
We need the following lemmas.
\begin{lemma}
	\label{lem:unif.loc.int}
	Let $\varphi$ be a $C_c^\infty(\C)$ function. For any $\varepsilon>0$, there exists $\delta>0$ independent of $N$ such that whenever $E$ is a Lebesgue measurable set satisfying $m(E)<\delta$, we have
	\begin{align}
		\left\vert \int_E \varphi(\lambda) \frac{1}{N}\Tr[\log\vert Y-\lambda\vert]\,d^2\lambda \right\vert &<\varepsilon \label{eq:Y.loc.int}\\
		\left\vert \int_E \varphi(\lambda) \tau[\log\vert y-\lambda\vert]\,d^2\lambda \right\vert &<\varepsilon \label{eq:y.loc.int}.
	\end{align}
	In other words, the logarithmic potentials of $Y$ and $y$ are uniformly locally integrable for all $N$.
\end{lemma}
\begin{proof}
	Let $\varepsilon>0$ be given. The function $\lambda\mapsto \log\vert z-\lambda\vert$ are uniformly locally integrable for all $\vert z\vert\leq 2M$; that is, there exists $\delta>0$ such that 
	\begin{equation}
		\label{eq:log.unif.loc.int}
		\left\vert \int_E \varphi(\lambda)\log\vert z-\lambda\vert \,d^2\lambda \right\vert <\varepsilon
	\end{equation}
	whenever $m(E)<\delta$ and $\vert z\vert\leq M$. By replacing $z$ in \eqref{eq:LogPotx} by $y$, an application of Fubini--Tonelli theorem shows that, with the same $\delta>0$, 
	\begin{align*}
		\left\vert\int_E \varphi(\lambda) \tau[\log\vert y-\lambda\vert]\,d^2\lambda  \right\vert= \int_{\vert z\vert\leq 2M}\left\vert\int_E \varphi(\lambda) \log\vert z-\lambda\vert d^2\lambda \right\vert\,d\mu_y(z)< \varepsilon
	\end{align*}
	whenver $m(E)<\delta$. We have used the fact that $\Vert y\Vert \leq 2M$.

	Now, we note that
	\[\frac{1}{N}\Tr[\log\vert Y-\lambda\vert] = \frac{1}{N}\sum_{j=1}^N\log\vert \lambda_j-\lambda\vert\]
	where $\lambda_j$ are the eigenvalues of $Y$. By assumption, $\|Y\|\leq \|A\|+\|\Sigma\|\leq 2M$, we must have $\max_{j}\vert\lambda_j\vert\leq 2M$. By \eqref{eq:log.unif.loc.int},
	\[\left\vert \int_E \varphi(\lambda)\frac{1}{N}\Tr[\log\vert Y-\lambda\vert]\,d^2\lambda\right\vert \leq \frac{1}{N}\sum_{j=1}^N\left\vert \int_E \varphi(\lambda)\log\vert \lambda_j-\lambda\vert \,d^2\lambda \right\vert <\varepsilon\]
	whenever $m(E)<\delta$. This shows \eqref{eq:Y.loc.int} and establishes the lemma.
\end{proof}
\begin{lemma}
	\label{lem:bounded0t}
	Fix a compact set $K\subset \Omega(T,a)^c$ and a bounded Borel function $\varphi$ on $\C$. Given any $\varepsilon>0$, there exists $t>0$ such that
	\[\int_K \vert \varphi(\lambda)\vert\int_{[0,t]}\vert \log\vert x\vert\vert\,d\mu_{\vert y-\lambda\vert}(x)d^2\lambda<\varepsilon.\]
\end{lemma}
\begin{proof}
	By Proposition~\ref{prop:boundedCauchy}, $G_{\widetilde{\mu}_{\vert y-\lambda\vert}}(ix)$ is uniformly bounded for all $\lambda\in K$ and $x\geq 0$. By \cite[Lemma 15]{GuionnetKZ-single-ring}, there exists a constant $C>0$ such that
	\[\widetilde{\mu}_{\vert y-\lambda\vert}([-x,x])\leq 2x \operatorname{Im}G_{\widetilde{\mu}_{\vert y-\lambda\vert}}(ix)\leq C x.\]
	By \cite[Lemma 4.1(a)]{BenaychGeorges2017}, there is a constant $C'>0$ such that
	\[\int_{[0,t]}\vert \log\vert x\vert\vert^2\,d\mu_{\vert y-\lambda\vert}(x)\leq C't(1-\log t).\]
	It follows that
	\begin{align*}
		\int_K \vert\varphi(\lambda)\vert\int_{[0,t]}\vert \log\vert x\vert\vert\,d\mu_{\vert y-\lambda\vert}(x)d^2\lambda \leq C'\|\varphi\|_{\infty}m(K) \sqrt{t(1-\log t)}
	\end{align*}
	where $m(K)$ is the Lebesgue measure of $K$. Now, it is evident that we can choose $t$ small enough such that the conclusion of the lemma holds.
\end{proof}

\begin{lemma}
	\label{lem:DiffsigmaSigma}
	Fix any $K\subset D(T,a)$. There exists a constant $C>0$ such that whenever $w\in K$, 
	\[\left\vert\int_\R\log\vert u\vert\,d\mu_{\sigma,\vert a-w\vert}(u)  - \int_\R\log\vert u\vert\,d\mu_{\Sigma,\vert A-w\vert}(u)\right\vert \leq C (\mathrm{d}_L(\widetilde\mu_\sigma ,\widetilde\mu_\Sigma) + \mathrm{d}_L(\widetilde\mu_{\vert a-w\vert}, \widetilde\mu_{\vert A-w\vert} ))\]
	for all $N$ large enough.
	\end{lemma}
	\begin{proof}
	We use the following observation due to \cite{TaoVu2010-aop} (see also (3.6) and (3.7) from \cite{BaoES2019singlering}) to write
	\begin{equation}
	   \label{eqn:IntCauchy-1}
		\begin{aligned}
			\int_\R\log\vert u\vert\,d\mu_{\sigma,\vert a-w\vert}(u) &= \int_\R \log\vert u-i\vert\,d\mu_{\sigma,\vert a-w\vert}(u) + \operatorname{Im}\int_0^{1} G_{\sigma,\vert a-w\vert}(i\eta)\,d\eta.\\
		\int_\R\log\vert u\vert\,d\mu_{\Sigma,\vert A-w\vert}(u)& = \int_\R \log\vert u-i\vert\,d\mu_{\Sigma,\vert A-w\vert}(u) + \operatorname{Im}\int_0^{1} G_{\Sigma,\vert A-w\vert}(i\eta)\,d\eta.
		\end{aligned}
	\end{equation}
	Since $\log|u-i|$ is a smooth function,
	by applying the continuity of free convolution \cite[Theorem 4.13]{BercoviciVoiculescu1993}, we have  
	\begin{equation}
	  \label{eqn:term2-a}
		\begin{aligned}
	& \left\vert	\int_\R \log\vert u-i\vert\,d\mu_{\sigma,\vert a-w\vert}(u)  
			 -\int_\R \log\vert u-i\vert\,d\mu_{\Sigma,\vert A-w\vert}(u)\right\vert\\
		  &\qquad\qquad\qquad\leq C \mathrm{d}_L( \widetilde\mu_{\sigma,\vert a-w\vert},\widetilde \mu_{\Sigma,\vert A-w\vert}  )
		  \leq C (\mathrm{d}_L(\widetilde\mu_\sigma ,\widetilde\mu_\Sigma) + \mathrm{d}_L(\widetilde\mu_{\vert a-w\vert}, \widetilde\mu_{\vert A-w\vert} )).
		\end{aligned}
	\end{equation}
	By the local stability of Cauchy transform in the bulk \cite[Theorem 2.7]{BaoES2016-jfa}, we have 
	\begin{equation}
	  \label{eqn:term2-b}
		\begin{aligned}
		&\left\vert \int_0^{1} G_{\sigma,\vert a-w\vert}(i\eta)\,d\eta
		 - \int_0^{1} G_{\Sigma,\vert A-w\vert}(i\eta)\,d\eta \right\vert\\
		&\qquad\qquad\qquad \leq C (\mathrm{d}_L(\widetilde\mu_\sigma ,\widetilde\mu_\Sigma) + \mathrm{d}_L(\widetilde\mu_{\vert a-w\vert}, \widetilde\mu_{\vert A-w\vert} )),
		\end{aligned}
	\end{equation}
	for any $w$ such that $0\in \mathcal{B}_{\mu_1\boxplus\mu_2}$, where $\mu_1=\widetilde{\mu}_{|a-w|}$ and $\mu_2=\widetilde{\mu}_{|\sigma|}=\widetilde{\mu}_{|T|}$.
	  Note that there exists $\varepsilon>0$ such that $K\subset D^{(\varepsilon)}(T,a)$. By the proof of \cite[Theorem 2.7]{BaoES2016-jfa} (see also \cite[Lemma 3.4]{BaoES2016-jfa} and \cite[Lemma 5.1]{BaoES2016-jfa}), the constant $C$ in \eqref{eqn:term2-a} and \eqref{eqn:term2-b} can be chosen independently from $w\in K$.
	\end{proof}

	\begin{lemma}
		\label{lem:unifLevy}
		Let $K\subset \C$ be a compact set. Then
		\[\mathrm{d}_L(\widetilde\mu_{\vert a-w\vert},\widetilde\mu_{\vert A-w\vert})\to 0\quad \mathrm{as}\; N\to\infty\]
		uniformly in $w\in K$.
	\end{lemma}
\begin{proof}
	Since $w\mapsto \mathrm{d}_L(\widetilde\mu_{\vert a-w\vert},\widetilde\mu_{\vert A_N-w\vert})$ is a continuous function in $w$, we can choose $w_N = \arg\max_{w\in K} \mathrm{d}_L(\widetilde\mu_{\vert a-w\vert},\widetilde\mu_{\vert A_N-w\vert})$. We claim that
\[\mathrm{d}_L(\widetilde\mu_{\vert a-w_N\vert},\widetilde\mu_{\vert A_N-w_N\vert})\to 0.\]
Suppose that the claim is false. There is a subsequence $w_{N_k}$ of $w_N$ such that \[\mathrm{d}_L(\widetilde\mu_{\vert a-w_{N_k}\vert},\widetilde\mu_{\vert A_{N_k}-w_{N_k}\vert})\geq \varepsilon_0\] for some $\varepsilon_0>0$. By dropping to a subsequence if necessary, we may assume $w_{N_k}\to w_0$ for some $w_0\in K$. However, this contradicts that 
\[\mathrm{d}_L(\widetilde\mu_{\vert a-w_{N_k}\vert},\widetilde\mu_{\vert A_{N_k}-w_{N_k}\vert})\to 0\]
because the opereator norms of 
\[\begin{pmatrix}
	0 & A_{N_k}-w_{N_k}\\
	(A_{N_k}-w_{N_k})^* & 0
\end{pmatrix}\quad \mathrm{and}\quad \begin{pmatrix}
	0 & a-w_0\\
	(a-w_0)^* & 0
\end{pmatrix}\]
are uniformly bounded and $\widetilde\mu_{\vert A_{N_k}-w_{N_k}\vert}\to\widetilde\mu_{\vert a-w_{0}\vert}$ in moments. This proves the lemma.
\end{proof}

\begin{proof}[Proof of Theorem~\ref{thm:deformedsinglering}]
	Let $f$ be a $C_c^\infty(\C)$ smooth function and write $K$ to be the support of $f$. Let $\varepsilon>0$. We want to show that
	\begin{equation}
		\label{eq:DSRest}
		\bbP\left(\left\vert \int_K(\Delta f)(\lambda) \frac{1}{N}\Tr[\log\vert Y-\lambda\vert]\,d^2\lambda - \int_K(\Delta f)(\lambda) \tau[\log\vert y-\lambda\vert]\,d^2\lambda \right\vert >4\varepsilon\right) \to 0.
	\end{equation}

	Recall that $S(T,a)$ is a finite set. If Assumption~(\ref{StrongAssumption2}) holds, set $E = \emptyset$; otherwise, take $E\subset\Omega(T,a)^c$ to be satisfying Assumption~(\ref{StrongAssumption1}). Since $\log$-potentials are uniformly locally integrable in the sense of Lemma~\ref{lem:unif.loc.int}, there exists an open set $W_1$ such that $E\cup S(T,a)\subset W_1$ and 
	\begin{equation}
		\label{eq:W1}
		\left\vert \int_{W_1}(\Delta f)(\lambda) \frac{1}{N}\Tr[\log\vert Y-\lambda\vert]\,d^2\lambda \right\vert+\left\vert \int_{W_1}(\Delta f)(\lambda) \tau[\log\vert y-\lambda\vert]\,d^2\lambda \right\vert <\varepsilon.
	\end{equation}
	By doing a compact exhaustion, there exist compact sets $F^{(1)}\subset F^{(2)}\subset \ldots\subset \Omega(T,a)$ such that 
	\[\bigcup_{k=1}^\infty F^{(k)} = \Omega(T,a).\]
	Since $\Omega(T,a)$ is a bounded set, by Lemma~\ref{lem:unif.loc.int}, there exists $n$ such that $\Omega(T,a)\setminus F^{(n)}$ has sufficiently small Lebesgue measure and \eqref{eq:W1} holds with $\Omega(T,a)\setminus F^{(n)}$ in place of $W_1$. We decompose $(K\cap W_1^c)\cap \Omega(T,a)$ into the disjoint union $K_1\cup W_2$ where 
\[
 K_1 = K\cap W_1^c\cap F^{(n)}\quad\text{and}\quad
		W_2 = K\cap W_1^c\cap (\Omega(T,a)\setminus F^{(n)}). 
\]
	Since $W_2\subset \Omega(T,a)\setminus F^{(n)}$, \eqref{eq:W1} holds with $W_2$ in place of $W_1$ in the equation. By writing $K_2 = (K\cap W_1^c)\cap \Omega(T,a)^c$, we decompose $K$ into the disjoint union
	\begin{equation}
		\label{eq:K-decomposition}
		K = (K\cap W_1)\cup W_2\cup K_1 \cup K_2.
	\end{equation}
	By construction, \eqref{eq:W1} holds with $K\cap W_1$ and $W_2$ in place of $W_1$ in the equation. In addition, $K_1\subset \Omega(T,a)$ and $K_2\subset \Omega(T,a)^c$ are compact sets. In the following paragraphs, it remains to estimate the integral on the left-hand side of \eqref{eq:DSRest} over $K_1$ and $K_2$ instead of $K$ separately.

	{\bf {Estimate over $K_1$:}} We first look at the integral on the left-hand side of \eqref{eq:DSRest} over $K_1$. We want to show that 
	\begin{equation}
		\label{eq:K1-integral}
		\bbP\left(\left\vert\int_{K_1}(\Delta f)(\lambda) \frac{1}{N}\Tr[\log\vert Y-\lambda\vert]\,d^2\lambda - \int_{K_1}(\Delta f)(\lambda) \tau[\log\vert y-\lambda\vert]\,d^2\lambda\right\vert > \varepsilon\right)\to 0.
	\end{equation}
		To this end, we will show that
		\begin{align}
		&\int_{K_1}(\Delta f)(\lambda)\left(\frac{1}{2N}\Tr\log\vert H^\lambda\vert-\int_\R\log\vert u\vert \,d\mu_{\Sigma,\vert A-\lambda\vert}\right)d^2\lambda\prec \frac{\|\Delta f\|_{L^1(\C)}}{N},\label{eq:K1est}\\
		&\int_{K_1}(\Delta f)(\lambda)\left(\int_\R\log\vert u\vert \,d\mu_{\Sigma,\vert A-\lambda\vert}d^2\lambda - \int_\R\log\vert u\vert \,d\mu_{\sigma,\vert a-\lambda\vert}d^2\lambda \right) \to 0.\label{eq:K1estFreeConv}
	\end{align}
	We first note that the convergence \eqref{eq:K1estFreeConv} is deterministic and it follows from Lemmas~\ref{lem:DiffsigmaSigma} and~\ref{lem:unifLevy}. 

	We now prove \eqref{eq:K1est}. For any $L>0$ and $\lambda\in K_1$, write the logarithmic potentials $\frac{1}{2N}\Tr\log\vert H^\lambda\vert$ and $\int_\R\log\vert u\vert \,d\mu_{\Sigma,\vert A-\lambda\vert}$ as in \eqref{eq:IntCauchy}. Since $K_1$ is compact, the estimate \eqref{eq:LogImag} holds uniformly for all $\lambda\in K_1$. Using \eqref{eq:IntCauchy} and \eqref{eq:LogImag}, in order to prove \eqref{eq:K1est}, it suffices to show
	\begin{equation}
		\label{eq:globalsecondterm}
		\left\vert \int_{K_1}(\Delta f)(\lambda)\left(\int_0^{N^L}(G^\lambda(i\eta)-G_{\Sigma,\vert A-\lambda\vert}(i\eta))d\eta\right)d^2\lambda\right\vert \prec \frac{\|\Delta f\|_{L^1(\C)}}{N}.
	\end{equation}
	The strategy is similar to that in Theorem~\ref{thm:localconv}. We decompose the integral with respect to $\eta$ into
	\[\int_0^{N^L} = \int_0^{N^{-L_1}}+\int_{N^{-L_1}}^{N^L}.\]
	Since $K_1$ is compact, there exists $\delta>0$ such that $K_1\subset D^{(\delta)}(T,a)$. By Theorem~\ref{thm:HlambdaCauchy}, we have the stochastic domination
	\begin{align}
		&\left\vert \int_{K_1}(\Delta f)(\lambda)\left(\int_{N^{-L_1}}^{N^L}(G^\lambda(i\eta)-G_{\Sigma,\vert A-\lambda\vert}(i\eta))d\eta\right)d^2\lambda\right\vert \\
		&\prec \int_{K_1}\vert (\Delta f)(\lambda)\vert \left(\int_{N^{-L_1}}^{N^L}\frac{1}{N\eta}d\eta\right)d^2\lambda\prec\frac{\|\Delta f\|_{L^1(\C)}}{N}\label{eq:large.eta}.
	\end{align} 
	We now estimate the integral $\int_0^{N^{-L_1}}$. Lemma~\ref{lem:boundedintG} shows that we can choose $L_1$ large enough such that
	\begin{equation}
		\label{eq:free.integral.global}
		\left\vert\int_{K_1}(\Delta f)(\lambda)\left(\int_0^{N^{-L_1}} G_{\Sigma,\vert A-\lambda\vert}(i\eta)\,d\eta\right)d^2\lambda\right\vert\leq \frac{\|\Delta f\|_{L^1(\C)}}{N}.
	\end{equation}
	Meanwhile, since we assume $\Vert \Sigma^{-1}\Vert\leq N^\alpha$, we can apply Lemma~\ref{lem:singularvalueG} and use the Markov's inequality to get the approximation
	\begin{align}
		&\quad\bbP\left(\left\vert\int_{K_1}(\Delta f)(\lambda)\left(\int_0^{N^{-L_1}}G^\lambda(i\eta)\,d\eta\right)d^2\lambda\right\vert>\frac{\|\Delta f\|_{L^1(\C)}}{N}\right)\nonumber\\
		&\leq \frac{N}{\|\Delta f\|_{L^1(\C)}}\int_{K_1}\vert(\Delta f)(\lambda)\vert\cdot\E\left\vert\int_0^{N^{-L_1}}G^\lambda(i\eta)\,d\eta\right\vert d^2\lambda\nonumber\\
		&\leq N^{-cL_1/2+\widetilde{c}+1}\label{eq:Markov.global}
	\end{align}
	for some positive constants $c$ and $\widetilde{c}$. By choosing large enough $L_1$, we can then combine \eqref{eq:large.eta}, \eqref{eq:free.integral.global} and \eqref{eq:Markov.global} to prove \eqref{eq:globalsecondterm}. This proves \eqref{eq:K1est}. By combining~\eqref{eq:K1estFreeConv}, we see that~\eqref{eq:K1-integral} holds.

 {\bf {Estimate over $K_2$:}} 	We now proceed to estimate the integral on the left-hand side of \eqref{eq:DSRest} over $K_2$ in place of $K$. If Assumption~(\ref{StrongAssumption1}) or~(\ref{StrongAssumption2}) holds, without loss of generality, we assume $0<\kappa_1<1/8$. We want to show that under one of these assumptions, there exists $\kappa_3, \kappa_4>0$ such that for all $N$ large enough, we have
	\begin{equation}
		\label{eq:EHbounded}
		\vert \E G^{\lambda}(i\eta)\vert \leq \kappa_4\quad \textrm{for all $\eta>N^{-\kappa_3}$ and $\lambda\in K_2$.}
	\end{equation}
	 We first assume Assumption~(\ref{StrongAssumption1}) holds. Recall that, by our decomposition~\eqref{eq:K-decomposition} of $K$, the set $E$ in Assumption~(\ref{StrongAssumption1}) is contained in $W_1$, hence disjoint from $K_2$. In this case, by Theorem~\ref{thm:ApproxSubordination}, there are functions $\omega_A$ and $r_A$ such that
	\[\E G^{\lambda}(i\eta) = \E G_{\widetilde{\mu}_{\vert A-\lambda\vert}}(\omega_A(i\eta))+r_A(i\eta).\]
	We apply Assumption~(\ref{StrongAssumption1}) with $S=K_2$. Fix a $\kappa_3$ such that $0<\kappa_3<\kappa_1<1/8$. If $N$ is large enough, by the estimate of $r_A(i\eta)$ in Theorem~\ref{thm:ApproxSubordination} and Lemma~\ref{lem:SImag}, $\omega_A(i\eta)$ is purely imaginary with $\operatorname{Im}\omega_A(i\eta)>N^{-\kappa_1}$ for all $\eta>N^{-\kappa_3}$ and $\lambda\in K_2$; moreover, $r_A(i\eta) \to 0$ uniformly in $\eta>N^{-\kappa_3}$ and $\lambda\in K_2$. Assumption~(\ref{StrongAssumption1}) tells us that $\vert\E G_{\widetilde{\mu}_{\vert A-\lambda\vert}}(\omega_A(i\eta))\vert\leq \kappa_2$; thus, \eqref{eq:EHbounded} holds for some $\kappa_4>0$. If Assumption~(\ref{StrongAssumption2}) holds instead of Assumption~(\ref{StrongAssumption1}), by applying Theorem~\ref{thm:ApproxSubordination} that there are functions $\omega_B$ and $r_B$ such that
	\[\E G^{\lambda}(i\eta) = \E G_{\widetilde{\mu}_{\sigma}}(\omega_B(i\eta))+r_B(i\eta),\]
	Eq. \eqref{eq:EHbounded} follows from a similar argument.

	By \eqref{eq:EHbounded} and \cite[Lemma 15]{GuionnetKZ-single-ring}, for all $\lambda\in K_2$,
	\[\E\widetilde{\mu}_{\vert Y-\lambda\vert}([-x,x])\leq 2\kappa_4 \max\{x, N^{-\kappa_3}\}.\]
	For any $t>0$, apply the Cauchy--Schwarz inequality to $\int_{[N^{-\kappa_3},t]}\vert \log\vert x\vert\vert\,d\mu_{\vert Y-\lambda\vert}(x)$; by \cite[Lemma 4.1(c)]{BenaychGeorges2017}, there exists a constant $\kappa_5>0$ such that
	\begin{equation}
		\label{eq:N-kappa3t}
		\E\int_{[N^{-\kappa_3},t]}\vert \log\vert x\vert\vert\,d\mu_{\vert Y-\lambda\vert}(x) \leq \kappa_5\sqrt{t\vert\log t\vert^2+N^{-\kappa_3}\vert \log N\vert^2}.
	\end{equation}
	Furthermore, given any $L_1>\kappa_3$ which will be chosen later,
	\begin{align}
		\E\int_{[N^{-L_1},N^{-\kappa_3}]}\vert \log\vert x\vert\vert\,d\mu_{\vert Y-\lambda\vert}(x) &\leq L_1\log N \,\E[\widetilde{\mu}_{\vert Y-\lambda\vert}([-N^{-\kappa_3},N^{-\kappa_3}])] \nonumber\\
		&\leq 2\kappa_4 L_1 N^{-\kappa_3}\log N\label{eq:N-L1-kappa3}.
	\end{align}
	We now choose $L_1>0$. Denote by $s_{\mathrm{min}}^\lambda$ the least singular value of $Y-\lambda$. We compute
	\begin{align*}
		\E \int_{[0,N^{-L_1}]}\vert\log\vert x\vert \vert\,d\mu_{\vert Y-\lambda\vert}(x) &\leq \E[\vert\log s_{\mathrm{min}}^\lambda \vert\mathbbm{1}_{s_{\mathrm{min}}^\lambda\leq N^{-L_1}}]\nonumber\\
		& = \int_0^{N^{-L_1}} \bbP(s_{\mathrm{min}}^\lambda\leq t)\frac{1}{t}dt+\bbP(s_{\mathrm{min}}^\lambda\leq N^{-L_1})\vert \log N^{-L_1}\vert.
	\end{align*}
	The above estimates are valid for $\lambda\in K_2$. We then apply Theorem~\ref{thm:lsv}; there are positive constants $c$ and $c'$ such that for any $\lambda$,
	\begin{align}
		\E \int_{[0,N^{-L_1}]}\vert\log\vert x\vert \vert\,d\mu_{\vert Y-\lambda\vert}(x)&\leq \int_0^{N^{-L_1}}t^{c-1} N^{c'}\,dt+L_1 N^{c'-L_1} \log N\nonumber\\
		&=\left(\frac{1}{c}+L_1\log N\right)N^{c'-L_1}\label{eq:0N-L_1}.
	\end{align}
	Hence, we choose $L_1=\kappa_3+c'$. By~\eqref{eq:N-kappa3t}, \eqref{eq:N-L1-kappa3} and \eqref{eq:0N-L_1}, there is a constant $C>0$ such that for all $\lambda\in K_2$,
	\begin{equation}
		\label{eq:log0t}
		\E\int_{[0,t]}\vert \log\vert x\vert\vert\,d\mu_{\vert Y-\lambda\vert}(x)\leq  \kappa_5\sqrt{t\vert\log t\vert^2+N^{-\kappa_3}\vert \log N\vert^2} +C N^{-\kappa_3} \log N.
	\end{equation}
	By Lemma~\ref{lem:bounded0t}, given any $\delta>0$, we can find $t>0$ such that the integral
	\[\int_{K_2}\vert\Delta f(\lambda)\vert\int_{[0,t]}\vert \log\vert x\vert\vert\,d\mu_{\vert y-\lambda\vert}(x)d^2\lambda<\delta\]
	for all large enough $N$. This together with \eqref{eq:log0t} shows that, by Markov's inequality, for any $\delta'>0$, we can choose $t$ small enough such that
	\begin{equation}
		\label{eq:0tinProb}
		\bbP\left(\left\vert \int_{K_2}(\Delta f)(\lambda) \left(\int_{[0,t]}\log x\,d\mu_{\vert Y-\lambda\vert}(x)-  \int_{[0,t]}\log x\,d\mu_{\vert y-\lambda\vert}(x)\right)d^2\lambda \right\vert >\delta'\right) \to 0.
	\end{equation}
	Finally, recall that $\|Y\|\leq 2M$ and $K_2$ is a bounded set in $\C$. By the weak convergence of $\mu_{\vert Y-\lambda\vert}$ to $\mu_{\vert y-\lambda\vert}$ in probability and by the dominated convergence theorem, for any $t>0$ and $\delta'>0$.
	\begin{equation}
		\begin{aligned}
		\label{eq:tinftyinProb}
		&\bbP\left(\left\vert \int_{K_2}(\Delta f)(\lambda) \left(\int_{[t,\infty)}\log x\,d\mu_{\vert Y-\lambda\vert}(x)- \int_{[t,\infty)}\log x\,d\mu_{\vert y-\lambda\vert}(x)\right)d^2\lambda \right\vert >\delta'\right) \to 0.
		\end{aligned}
	\end{equation}
	Combining \eqref{eq:0tinProb} and \eqref{eq:tinftyinProb} with suitable $\delta'>0$, by triangle inequality, we see that
	\begin{equation}
		\label{eq:K2estimate}
		\bbP\left(\left\vert \int_{K_2}(\Delta f)(\lambda) \frac{1}{N}\Tr[\log\vert Y-\lambda\vert]\,d^2\lambda - \int_{K_2}(\Delta f)(\lambda) \tau[\log\vert y-\lambda\vert]\,d^2\lambda \right\vert >\varepsilon\right) \to 0.
	\end{equation}
	By our choice of $W_1$ and $W_2$, \eqref{eq:DSRest} follows from \eqref{eq:K1-integral} and \eqref{eq:K2estimate}. This also completes the proof of the theorem.
\end{proof}

\subsection{The Hermitian or unitary case}
\label{sect:hermitian-singlering}
The following thoerem shows that if $A$ is Hermitian or $A$ is unitary for all $N$, the empirical eigenvalue distribution of $Y$ in \eqref{eq:modelY} converges to the Brown measure of $y$ in \eqref{eq:limity} without additional assumptions on $A$ or $\Sigma$. By taking $A = 0$, the following theorem removes a regularity assumption of the single ring theorem by Guionnet et al. \cite{GuionnetKZ-single-ring}.

We do not need the extra assumptions on $A$ or $\Sigma$ as in Theorem~\ref{thm:deformedsinglering} because of the following property of normal matrices. If $\lambda_1,\ldots,\lambda_N$ are the eigenvalues of a normal matrix $A$, then the norm $\|(A-\lambda)^{-1}\|$ is given by 
\[\|(A-\lambda)^{-1}\| = \sup_{j}\vert \lambda_j-\lambda\vert^{-1}.\]
If $A$ is a Hermitian or unitary matrix for all $N$, then we are able to find an open set $W_1$ of arbitrarily small Lebesgue measure such that $\|(A-\lambda)^{-1}\|$ is bounded for all $\lambda\in W_1^c$ and for all $N$. Consequently, the estimate in Lemma~\ref{lem:singularvalueG} is valid for all $\lambda\in W_1^c$. We apply Lemma~\ref{lem:unif.loc.int} for the estimate of the log potentials for $\lambda\in W_1$.
\begin{theorem}
	\label{thm:SpecificA}
	Consider the random matrix model $Y$ and the operator $y\in\mathcal{A}$ as in \eqref{eq:modelY} and \eqref{eq:limity}. If $A$ is Hermitian or if $A$ is unitary for all $N$, then the empirical eigenvalue distribution of $Y$ converges weakly to the Brown measure of $y$ in probability. 
	
	More generally, if $A$ is a normal marix and there is a closed set $F$ with Lebesgue measure zero independent of $N$ such that all the eigenvalues of $A$ lie inside $F$ for all $N$, then the empirical eigenvalue distribution of $Y$ converges weakly to the Brown measure of $y$ in probability. 
\end{theorem}
\begin{proof}
	We prove the more general statement for normal matrix $A$ satisfying the hypothesis. Let $f$ be a $C_c^\infty(\C)$ smooth function and write $K$ to be the support of $f$. Let $\varepsilon>0$. We want to show~\eqref{eq:DSRest} under the hypothesis of this theorem. We decompose 
	\begin{equation}
		\label{eq:K-decomposition-Hermitian}
		K = (K\cap W_1)\cup W_2\cup K_1\cup K_2
	\end{equation}
	in the exact same way in~\eqref{eq:K-decomposition} with $F$ in place of $E$. We remark that, in contrast to the set $E$ in the proof of Theorem~\ref{thm:deformedsinglering}, this set $F$ may have nonempty intersection with $\Omega(T,a)$. The estimate~\eqref{eq:W1} holds with the $K\cap W_1$ or $W_2$ in~\eqref{eq:K-decomposition-Hermitian} in place of $W_1$. In the next paragraph, we will show that the estimates~\eqref{eq:K1-integral} and~\eqref{eq:K2estimate} hold with the $K_1$ and $K_2$ in~\eqref{eq:K-decomposition-Hermitian}.
	
	By our choice of $W_1$, the compact sets $K_1$ and $K_2$ have positive distance from $F$ which contains all the eigenvalues of $A$. For any $\lambda\in K_1\cup K_2$, the matrix $A-\lambda$ is invertible and $\Vert (A-\lambda)^{-1}\Vert\leq C$ for some constant $C>0$. In particular, we can apply Theorem~\ref{thm:lsv} to estimate the least singular value of $Y-\lambda$. The proof of the estimates~\eqref{eq:K1-integral} and~\eqref{eq:K2estimate} with the $K_1$ and $K_2$ in~\eqref{eq:K-decomposition-Hermitian} then follows from the exact same procedure in the proof of Theorem~\ref{thm:deformedsinglering}. Remark that in the process of proving~\eqref{eq:K2estimate}, \eqref{eq:EHbounded} holds with this $K_2$ because $K_2$ has a positive distance from all the eigenvalues of $A$.  
\end{proof}

\section{Example: the Jordan block matrix}
\label{sect:Jordan}
Consider a sequence of $N\times N$ Jordan block matrices 
\begin{equation}
	\label{eq:JordanAdef}
	A = \begin{pmatrix}
	0& 1& 0 &\cdots & 0 \\
	0& 0& 1 &\cdots & 0\\
	\vdots & \vdots& \vdots& \ddots &\vdots\\
	0&0&0&\cdots& 1\\
	0&0&0&\cdots &0
\end{pmatrix}.
\end{equation}
It is well-known that $A$ converges in $\ast$-distribution to a Haar unitary operator $a$ but the eigenvalue distribution converges to $\delta_0$, but not the Brown measure of the limit operator. As an application of Theorem~\ref{thm:deformedsinglering}, we will show that, however, the empirical eigenvalue distribution of the random matrix $U\Sigma V+A$ \emph{does} converge to the Brown measure of the operator of its limit in $\ast$-distribution. We will show that Assumption~(\ref{StrongAssumption1}) in Theorem~\ref{thm:deformedsinglering} holds for $A$. The main result of this section is Proposition~\ref{prop:JordanEV}. We first study the singular values of $A-\lambda$ for $\vert \lambda\vert\neq 1$.

\begin{lemma}
	\label{lem:JordanSV}
	For $\vert \lambda\vert\neq 1$, the magnitudes of the singular values of $A-\lambda$ can be described as follow.
	\begin{enumerate}
		\item If $\vert\lambda\vert<1$, all but one singular values of $A-\lambda$ are at least $1-\vert\lambda\vert$.
		\item If $\vert \lambda\vert >1$, all the singular values of $A-\lambda$ are at least $\vert\lambda\vert-1$.
	\end{enumerate}
\end{lemma}
We prove a more general result on estimating singular values. While we only state the result for matrices, the result, with the same proof, indeed holds for compact operators on a Hilbert space. For an $N\times N$ matrix $L$, we denote the singular values of $L$ in decreasing order by
\[s_1(L)\geq s_2(L)\geq\ldots\geq s_N(L).\]

\begin{lemma}
	\label{prop:finiterankpert}
	Suppose that $L$ and $U$ are $N\times N$ matrices such that $L-U$ has rank $k$. Then, for any $\lambda\in\C$,
	\[s_n(\lambda-L)\geq s_{n+k}(\lambda-U).\]
	for all $0\leq n \leq n+k\leq N$. In particular, if $U$ is unitary and $\vert\lambda\vert < 1$, \[s_1(\lambda-L)\geq \ldots\geq s_{N-k}(\lambda-L)\geq 1-\vert\lambda\vert.\]
\end{lemma}
\begin{proof}
	It is well-known that (for example, see \cite[Theorem 2.1]{GohbergKreinBook}), for any $n$, the $n$th-largest singular value of any matrix $L$ can be computed as
	\[s_n(L) = \inf\{\|L-X\|:\operatorname{rank}(X)<n\}.\]
	Since $L-U$ has rank $k$, $L-U+X$ has at most rank $n+k-1$ for any $X$ with rank at most $n-1$. We then apply the above formula to $\lambda-L$ and get
	\begin{align*}
		s_n(\lambda-L)& =  \inf\{\|\lambda-L-X\|:\operatorname{rank}(X)<n\}\\
		&=\inf\{\|(\lambda-U)-(L-U+X)\|:\operatorname{rank}(X)<n\}\\
		&\geq \inf\{\|(\lambda-U)-Z\|:\operatorname{rank}(Z)<n+k\}\\
		&= s_{n+k}(\lambda-U).
	\end{align*}
	For the last assertion, since the above computation shows $s_{N-k}(\lambda-L)\geq s_{N}(\lambda-U)$, it suffices to show that $s_{N}(\lambda-U)\geq 1-\vert\lambda\vert$. But $s_N(\lambda-U) = 1/\|(\lambda-U)^{-1}\|$ and if $U$ is unitary,
	\begin{align*}
		\|(\lambda-U)^{-1}\| &= \|(\lambda U^*-1)^{-1}U^*\|\\
		&\leq \sum_{l=0}^\infty\|\lambda U^*\|^l\\
		&=\frac{1}{1-\vert\lambda\vert}.
	\end{align*}
	This completes the proof of the lemma.
\end{proof}

\begin{proof}[Proof of Lemma~\ref{lem:JordanSV}]
	Let $E_{N1}$ be the matrix that has all entries equal to $0$ except the $(N,1)$-entry, in which the entry is $1$. Then $A+E_{N1}$ is a unitary matrix; in fact, this matrix cyclically permutes the standard basis. Point (1) of the lemma  follows from applying Lemma~\ref{prop:finiterankpert} with $L=A$, $U = A+E_{N1}$ and $k=1$.

	For Point (2), we first note that $\|A\| = 1$ and assume $\vert\lambda\vert> 1$. We expand $(\lambda-A)^{-1}$ into a power series to get the norm estimate
	\begin{align*}
		\|(\lambda-A)^{-1} \|&\leq \sum_{l=0}^\infty \frac{\|A\|^l}{\vert\lambda\vert^{l+1}}=\frac{1}{\vert\lambda\vert-1}.
	\end{align*}
	Since the least singular value of $\lambda-A$ is given by $1/\|(\lambda-A)^{-1}\|$, we conclude that all the singular values of $\lambda-A$ are at least $\vert\lambda\vert - 1$.
\end{proof}

\begin{proposition}
	\label{cor:ACauchy}
	Given any compact set $K$ satisfying either $K\subset \mathbb{D}$ or $K\subset \C\setminus \overline{\mathbb{D}}$, there exists a constant $\kappa>0$ such that
	\[G_{\widetilde\mu_{\vert A-\lambda\vert}}(i\eta)\leq \kappa\]
	for all $\eta>1/N$ and $\lambda\in K$.
\end{proposition}
\begin{proof}
	We look at the case that $K\subset\mathbb{D}$. Let $d$ be defined by
	\[d = 1-\sup_{\lambda\in K}\vert \lambda\vert.\]
	Since $K\subset\mathbb{D}$ is compact, we must have $d>0$. Now, by Lemma~\ref{lem:JordanSV}, $A-\lambda$ has $N-1$ singular values that are at least $d$. Hence, for any $\lambda\in K$,
	\begin{align*}
		G_{\widetilde\mu_{\vert A-\lambda\vert}}(i\eta) &\leq \frac{1}{N}\frac{\eta}{s_N(A-\lambda)^2+\eta^2} + \frac{N-1}{N}\frac{\eta}{d^2+\eta^2}\\
		&\leq \frac{1}{N}\frac{1}{\eta}+\frac{1}{2d}\\
		&< 1+\frac{1}{2d}.
	\end{align*}
	This shows that, in this $K\subset \mathbb{D}$, $G_{\widetilde\mu_{\vert A-\lambda\vert}}(i\eta) < 1+1/2d$ for all $\eta>1/N$ and $\lambda\in K$.

	The case $K\subset\C\setminus\overline{\mathbb{D}}$ is simpler. Since all the singular values of $A-\lambda$ are at least $d'-1$, which is positive, where $d' = \inf_{\lambda\in K}\vert \lambda\vert$. This shows that, in this case,
	\[G_{\widetilde\mu_{\vert A-\lambda\vert}}(i\eta)\leq \frac{1}{2(d'-1)}\]
	for all $\eta>0$ and $\lambda\in K$.
\end{proof}

\begin{proposition}
	\label{prop:JordanEV}
	Consider the random matrix model $Y$ as in \eqref{eq:modelY} and the operator $y\in\mathcal{A}$ as in \eqref{eq:limity} with $A$ defined in \eqref{eq:JordanAdef}. Suppose that $\Sigma$ is invertible and $\Vert \Sigma^{-1}\Vert\leq N^\alpha$ for some $\alpha>0$ independent of $N$. Then the limiting eigenvalue distribution of $Y$ is the Brown measure of $y$.
\end{proposition}
We run computer simulations with $N=2000$. In the simulation, $\Sigma$ is a deterministic $N\times N$ diagonal matrix such that half of the diagonal entries are $1$ and half of the diagonal entries are $N^{-5}$. Figure~\ref{fig:compareUnitary} shows computer simuations of the eigenvalues of $U\Sigma V+A$ and $U\Sigma V+W$, where $W$ is a Haar-distributed unitary random matrix independent of $U$ and $V$. Both $A$ and $W$ converge in $\ast$-distribution to a Haar unitary operator. We can see that the eigenvalues distributions of the two simulations are about the same.
\begin{proof}
	Recall that the matrix $A$ converges to a Haar unitary operator $a$ in $\ast$-distribution as $N\to\infty$. By Proposition~\ref{cor:ACauchy}, Assumption~(\ref{StrongAssumption1}) in Theorem~\ref{thm:deformedsinglering} holds if we take $E = \mathbb{T}$. The conclusion of this proposition then follows from Theorem~\ref{thm:deformedsinglering}.
\end{proof}
\begin{figure}
	\includegraphics[width = 0.48\textwidth]{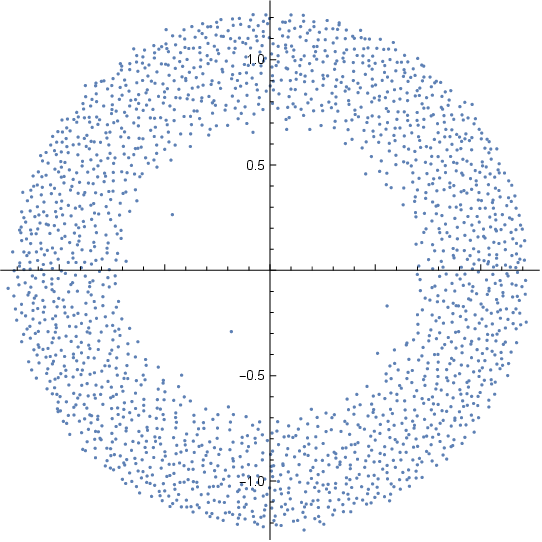}
	\includegraphics[width = 0.48\textwidth]{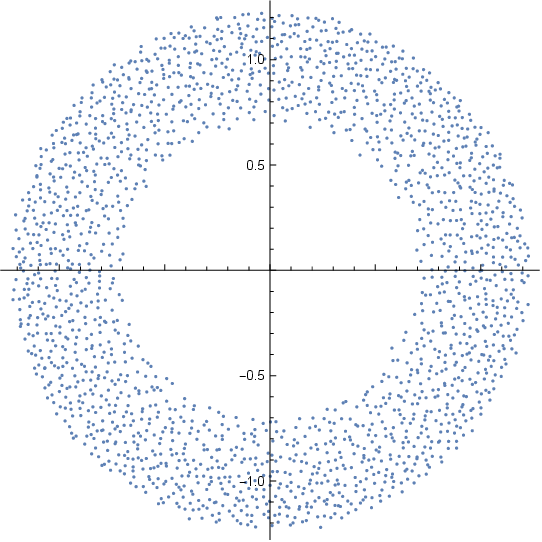}
	\caption{Eigenvalue simulations of $U\Sigma V+A$ (left) and $U\Sigma V+W$ (right) where $A$ is the Jordan block matrix, $\Sigma$ is a deterministic diagonal matrix such that $\mu_{\Sigma} \approx (1/2)(\delta_0+\delta_1)$, and $W$ is a Haar-distributed unitary matrix independent of $U$ and $V$ \label{fig:compareUnitary}}
\end{figure}

\subsection*{Acknowledgments}
The authors thank Hari Bercovici for useful conversations. We particularly thank him for showing us the proof of Lemma~\ref{prop:finiterankpert}. The first author is partially supported by the MoST grant 111-2115-M-001-011-MY3. The second author is supported in part by Collaboration Grants for Mathematicians from Simons Foundation, NSF grant LEAPS-MPS-2316836,
and NSF CAREER Award DMS-2339565.

\bibliographystyle{acm}
\bibliography{BrownMeasure}

\end{document}